\documentclass[11pt]{amsart}
\usepackage{}
\usepackage{amsmath,amssymb,amsthm}
\usepackage[latin1]{inputenc}
\usepackage{version, tabularx, multicol}
\usepackage{graphicx,float}
\usepackage{stmaryrd}
\usepackage{color,latexsym,amsfonts}

\theoremstyle{plain}
\newtheorem{theo*}{Theorem}
\newtheorem{cor*}[theo*]{Corollary}
\newtheorem{prop*}[theo*]{Proposition}
\newtheorem{theo}{Theorem}[section]
\newtheorem{cor}[theo]{Corollary}
\newtheorem{prop}[theo]{Proposition}
\newtheorem{lem}[theo]{Lemma}

\theoremstyle{remark}
\newtheorem{rem}[theo]{Remark}

\newcommand{\ca}{{\mathcal A}}
\newcommand{\ck}{{\mathcal K}}
\newcommand{\cl}{{\mathcal L}}
\newcommand{\cn}{{\mathcal N}}
\newcommand{\cm}{{\mathcal M}}
\newcommand{\Tau}{{\mathcal T}}

\newcommand{\D}{{\mathbb D}}
\newcommand{\E}{{\mathbb E}}
\newcommand{\N}{{\mathbb N}}
\renewcommand{\P}{{\mathbb P}}
\newcommand{\R}{{\mathbb R}}

\newcommand{\rE}{{\rm E}}
\newcommand{\rN}{{\rm N}}
\newcommand{\rP}{{\rm P}}

\newcommand{\ind}{{\bf 1}}

\newcommand{\Li}{\text{Li}}
\newcommand{\lam}{\lambda}
\newcommand{\ebt}{\mathrm{e}^{2 \beta \theta t}}
\newcommand{\ebs}{\mathrm{e}^{2 \beta \theta s}}
\newcommand{\embt}{\mathrm{e}^{-2 \beta \theta t}}
\newcommand{\embs}{\mathrm{e}^{-2 \beta \theta s}}
\newcommand{\eb}[1]{\mathrm{e}^{2 \beta \theta #1}}
\newcommand{\emb}[1]{\mathrm{e}^{-2 \beta \theta #1}}
\newcommand{\reff}[1]{(\ref{#1})}

\newcommand{\inv}[1]{\mathop{\frac{1}{ #1}}\nolimits}
\newcommand{\expp}[1]{\mathop {\mathrm{e}^{ #1}}}

\title[Total length of the genealogical tree for  quadratic CB]
 {Total length of the genealogical tree for  quadratic
 stationary  continuous-state  branching    processes}

\date{\today}
\author{Hongwei Bi}
\address{
Hongwei Bi,
School of Insurance and Economics, University of International Business
and Economics, Beijing 100029, China
}
\email{bihw2009@gmail.com}

\author{Jean-Fran\c cois Delmas}

\address{
Jean-Fran\c cois Delmas,
Université Paris-Est, CERMICS (ENPC), F-77455
Marne La Vallee, France.}

\email{delmas@cermics.enpc.fr}

\begin{document}


\begin{abstract}
  We  prove  the   existence  of  the  total  length   process  for  the
  genealogical tree of   a population model with random  size given by a
  quadratic  stationary continuous-state  branching  processes. We  also
  give, for  the one-dimensional marginal, its Laplace transform as well
  as the fluctuation  of the corresponding  convergence.     This result
  is  to  be  compared  with   the  one  obtained  by  Pfaffelhuber  and
  Wakolbinger  for constant  size population  associated to  the Kingman
  coalescent. We  also give  a time reversal  property of the  number of
  ancestors process at all time, and give a description of the so-called
  lineage tree in this model.
\end{abstract}

\maketitle

\section{Introduction}

\subsection{The model}
Stochastic models for the evolution of a stationary population goes back
to the Wright-Fisher model, which  is for a finite fixed size population
in discrete generations.  Fleming-Viot  processes extend those models to
infinite size population  (with infinitesimal individuals) in continuous
time, see \textsc{Donnelly} and \textsc{Kurtz} \cite{dk:crfvmvd}. On the
other hand, the Galton-Watson process models the evolution of a discrete
random-size population  in discrete  generations based on  the branching
property:  descendants  of  two  individuals in  the  same  generation
behaves independently. Continuous  state branching (CB) processes extend
those   models   to  infinite   size   population  (with   infinitesimal
individuals) in continuous  time. The description of the  genealogy of a
CB  process  is done  using  historical  Dawson-Watanabe processes,  see
\textsc{Donnelly}  and \textsc{Kurtz}  \cite{dk:prmvpm}, or  Lévy trees,
see \textsc{Duquesne} and  \textsc{Le Gall} \cite{dlg:rtlpsbp}. In order
to  consider  Galton-Watson  processes  or CB  processes  in  stationary
regime,   one    has   to   condition   them    on   non-explosion   and
non-extinction. Then  one gets Galton-Watson process or  CB process with
an   immortal  individual,   see  \textsc{Delmas}   and  \textsc{Hénard}
\cite{dh:wdsds}  in  this   direction  for  non-homogeneous  models  and
references therein.  This  can also be seen as  Galton-Watson process or
CB process  with immigration if one removes the immortal  individual. We
shall consider one of the simplest model developed in \textsc{Chen} and
\textsc{Delmas}  \cite{cd:spsmrcatsbp} of  CB process  with  an immortal
individual  which  corresponds  to  a quadratic  sub-critical  branching
mechanism. The  results we present  concern neutral populations.

\subsection{The genealogy}
Describing  the genealogy  of a  large population  is a  key  issue in
population genetics.  A well-established model in this  direction is the
Kingman  coalescent  which  describes the genealogy  of  a  Fleming-Viot
process. Intuitively we  may think of the Kingman  coalescent at some
fixed time $s$ as a random tree with infinitely many leaves (corresponding
to the  individuals alive at time  $s$), when backwards in  time any two
lineages  coalesce   independently  at  rate   1.   See  \textsc{Pitman}
\cite{p:cmc}   and  \textsc{Sagitov}   \cite{s:gcamal}  for   a  general
description of  the coalescent processes.

The  study of  the evolution  in  $t$ of  the genealogical  tree of  the
population  at time  $t$,  or of  some  of its  functional has  recently
attracted  some interest  in mathematical  population genetics.  In this
direction  for the  quadratic  Fleming-Viot process  (associated to  the
Kingman  coalescent),  see  \textsc{Greven},  \textsc{Pfaffelhuber}  and
\textsc{Winter}  \cite{gpw:trdmpa}. The  functional of  the genealogical
tree of interest are:
\begin{itemize}
\item The time to the most recent common ancestor (TMRCA) at time $t$ is
  the distance between any leaf (which are all living individual at time
  $t$) of the genealogical  tree and its root. \textsc{Pfaffelhuber} and
  \textsc{Wakolbinger}  \cite{pw:pmrcaec} studied  the  evolving Kingman
  coalescent case and \textsc{Evans} and \textsc{Ralph} \cite{er:dtmrca}
  the large branching population case.

\item  The number of  mutations observed  in a  population in  a neutral
  model is distributed  as a Poisson random variable  with mean the rate
  of  mutation times  the total  tree  length of  the genealogical  tree
  (other  similar quantities  of interest  are the  number  of mutations
  which  appear only  once;  this  is distributed  as  a Poisson  random
  variable with mean the rate of  mutation times the total length of the
  external branches of the genealogical tree).  This motivated the study
  of the  rescaled total length  of coalescent trees which  converges in
  distribution to the Gumbel distribution  at a given fixed time for the
  Kingman  coalescent. (See  \textsc{Janson} and       \textsc{Kersting}
  \cite{jk:telkc}             for the external length asymptotics.)  The
   corresponding                             limiting process  has  been
  studied    in    \textsc{Pfaffelhuber},    \textsc{Wakolbinger}    and
  \textsc{Weisshaupt}  \cite{pww:tlec}      as well  as \textsc{Dahmer},
  \textsc{Knobloch}  and \textsc{Wakolbinger}  \cite{dkw:ktlphiqv} where
  it is proved the limiting process is not a semi-martingale.

Extension has been provided for other $\Lambda$-coalescents, see
\textsc{Kersting}, \textsc{Schweinsberg} and \textsc{Wakolbinger}
\cite{ksw:ebc} for Beta-coalescent and
\textsc{Schweinsberg} \cite{s:debsc} for the Bolthausen-Sznitman
coalescent.
\end{itemize}

Our main objective is to study the limit process of the renormalized
total length of the genealogical tree in  a population with random size
given by a quadratic stationary CB process.

\subsection{Main results}

We model  the random size  of the population  at time $t$ by  $Z_t$ with
$(Z_t, t\in \R)$  a stationary CB (or CB process with immigration) process with
sub-critical   quadratic   branching   mechanism.   This   model,    see
\cite{cd:spsmrcatsbp}  or   Section  \ref{sec:pop-mod}  for   a  precise
definition,  is characterized  by two  positive parameters  $\theta$ and
$\beta$, which describe the mean size of the population and a time scale:
\begin{itemize}
\item The  random size of the  population, $Z_t$, is  distributed as the
  sum  of  two  independent   exponential  random  variables  with  mean
  $1/(2\theta)$.
\item The TMRCA of the population living at time $t$, $A_t$, is
  distributed as the maximum of two  independent exponential
random variables  with mean $1/(2\beta\theta)$.
\end{itemize}
In particular, we have:
\[
\E[Z_t]=\inv{\theta} \quad\text{and}\quad
\E[A_t]=\frac{3}{4\beta} \E[Z_t].
\]

\begin{figure}[!ht]
\begin{center}
\includegraphics[width=9cm,height=6cm]{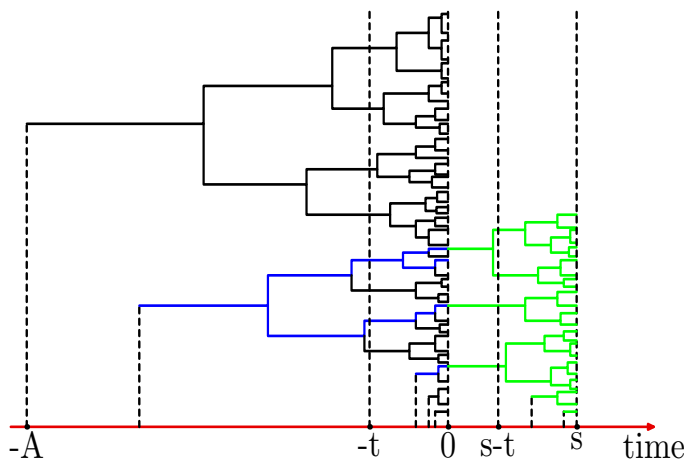}
\end{center}
\caption{We represent: in red the line of the immortal individual;
in blue and green the  genealogical tree of the population living at
time  $s$; in blue and black the genealogical tree of the population living at
time  $0$. At $t>0$ units of time before present, the number of
ancestors is $M_{-t}^0=8 $ for the population living at time 0 and
$M_{s-t}^s=4$ for the population living at time $s$. The TMRCA of the
population living at time 0 is $A$.}
\label{fig:gen}
\end{figure}

For $s<t$  let $M_s^t$  be the number of ancestors  at time
$s$  of the population living  at time $t$, the immortal  individual being excluded, see  \reff{def:mrt} for a precise definition. The following time
reversal  property for  the number  of ancestors  process $(M_{s}^{s+r},
s\in \R, r>0)$, see  Theorem \ref{theo:time-reversal}, is similar to the
time reversal property of the look-down process in the Kingman case, see
also Lemma  8 from \textsc{Aldous}  and \textsc{Popovic} \cite{ap:cbpmb}
in a  critical branching process setting  at a fixed time.  The proof of
the  next  Theorem  does  not  rely  on  discrete  approximation  as  in
\cite{ap:cbpmb}.

\begin{theo*}[Time reversal property]
The process  $(M_{s-r}^{s}, s\in \R, r>0)$ is distributed as
$(M_{s}^{s+r}, s\in \R, r>0)$.
\end{theo*}

We define for $r>0$ the ``probability'' of
an  infinitesimal  individual  to  have  descendants $r$  unit  of  time
forward, see  definition \reff{eq:defc}, as:
\[
c(r)=\frac{2\theta}{\expp{2\beta\theta r}  -1}
\cdot
\]
The lineage tree $\ca_s$ of the population at time $s$ is defined
by \textsc{Popovic} \cite{p:agcbp} (see also \cite{ap:cbpmb})  in a critical
branching setting (see also the references in Remark \ref{rem:AldPop}),
and it corresponds in our setting  to the jumping times of the process $(M_{s-r}^s, r>0)$:
\[
\ca_s=\{r>0;  M_{s-r}^s
-M_{(s-r)_-}^s =1\}.
\]
The lineage tree of  $Z_s$  at  some   current time $s$ is depicted
in Figure \ref{fig:gen}.
Using the time reversal property, we
deduce in Remark \ref{rem:AldPop} the following Corollary.

\begin{cor*}
  The  lineage  tree  $\ca_s$  has  the same  distribution  as  the  set
  $\{\zeta_j; x_j <Z_s\}$  where $\sum_{j\in J} \delta_{x_j,\zeta_j}(dx,
  dz)$ is a Poisson point measure on $(0,+\infty )^2$ with intensity $dx
  |c'(z)|dz$ and independent of $Z_s$.
\end{cor*}

The  process $(M_{s-r}^s,  r>0)$  is  a (forward)  death  process and  a
(backward)  birth  process whose  intensity  are  given in  Propositions
\ref{prop:M-mart}  and \ref{prop:M-back}.  We  also give  in Proposition
\ref{prop:ZfromM}  a  reconstruction result  of  the process  $(Z_{s-t},
t>0)$ from the process $(M_{s-r}^s,  r>0)$ by grafting CB processes, and
we  then deduce  a  formula on  the  weighted integral  of the  ancestor
process,  see  Corollary  \ref{cor:exp-M}.   For reconstructions  of  CB
processes  from   backbones  instead  of  genealogical   tree  see  also
\textsc{Duquesne} and \textsc{Winkel} \cite{dw:glt}.

The total length of the genealogical tree for the population
living at time $s$, up to time $s-\varepsilon$ (with $\varepsilon>0$) is
given by:
\[
L_\varepsilon^s= \int_\varepsilon^\infty M_{s-r}^s\, dr,
\]
and we consider the normalized  total length  up to time $s-\varepsilon$
defined by:
\[
\cl_\varepsilon^s=L_\varepsilon^s-Z_s\int_\varepsilon^\infty c(r)\, dr.
\]

We have the following result, see Theorems \ref{cl2conver} and
\ref{theo:cadlag} as well as Lemma \ref{laplace-W}.

\begin{theo*}
There exists a càdlàg stationary process $(W_s, s\in \R)$  such that for all $s\in
\R$ the compensated tree length  $(\cl_\varepsilon^s,
\varepsilon>0)$ converges a.s. and in $L^2$ to $W_s$.

Furthermore we have for $\lambda>0$:
\[
\rE\left[\expp{-2 \beta\theta\lambda W_0}\,|\,Z_0 =\frac{z}{2\theta}\right]
=\expp{-z\varphi(\lambda)}
\quad\text{and}\quad
\rE\left[\expp{-2\lambda\beta \theta  W_0}\right]
=\left(1+\varphi(\lambda)\right)^{-2}
\]
with
\[
\varphi(\lambda)=  -\lambda \int_0^1 dv\, \frac{1-v^{\lambda}}{1-v}\cdot
\]
\end{theo*}
Proposition \ref{prop:fluctW} gives the fluctuations:
$\sqrt{\beta } (\cl_\varepsilon-W_0)/\sqrt{\varepsilon}$ converges in
distribution as $\varepsilon$ goes down to 0 towards
$\sqrt{2Z_0}\,  G$,
with $G\sim \cn(0,1)$ a standard Gaussian random variable independent of
$Z_0$.

Notice the process $(\cl_\varepsilon^s, s\in \R)$ is not continuous, and
this implies that $W$ is not continuous.
We also provide the covariance of $W$, see Proposition \ref{prop:ewswo},  and get, see Remark
\ref{rem:cov}, that there exists some finite positive constant $C$
such that:
\[
\rE[(W_s-W_0)^2]
\sim_{0+} Cs \, \log(s)^2.
\]

\section{Population model}\label{sec2:pre}
Let $\beta>0$ and $\theta>0$ be fixed scale parameters.

\subsection{Sub-critical quadratic CB process}
Consider  a  sub-critical  branching  mechanism  $\psi(\lambda)  =  \beta
\lambda^2 + 2 \beta \theta \lambda $, let $\P_x$
be the  law of  a CB process $Y =  (Y_t, t \geq  0)$ started  at mass  $x$ with
branching mechanism $\psi$. We extend $Y$ on $\R$ by setting $Y_t=0$ for
$t<0.$ Let $\E_x$ and $\N$ be respectively the corresponding expectation
and the  canonical measure (excursion  measure) associated to  $Y$.
Recall that $Y$ is Markovian under $\P_x$ and $\N$. We have for
every $t > 0$:
\[
\E_x\left[\expp{-\lam Y_t}\right] = \expp{-x u(\lam, t)} \quad
\text{for}\quad
\lambda > - \frac{2\theta}{1-\embt}
\]
with
\begin{equation} \label{eq:defu}
u(\lam, t)  =  \N [ 1 - \expp {- \lam Y_t} ]
= \frac{ 2 \theta \lam}{( 2 \theta + \lam) \ebt - \lam}
\end{equation}
satisfying  the  backward  and    forward    equations:
\begin{equation}\label{eq:backe}
\partial_t u(\lam, t)=-\psi(u(\lam, t)) \quad \text{and}\quad
\partial_t u(\lam, t)= -\psi(\lam)\,
\partial_\lam u(\lam, t).
\end{equation}
Then it is easy to  derive that for $t>0$:
\begin{equation}
\label{eq:int-u}
\beta \int_0^t \!\!u(\lambda, r)\; dr= \log\left(1+ \lambda
  \frac{1-\expp{-2\beta\theta t }}{2\theta}\right)
\quad\text{and}\quad
\beta\int_0^\infty \!\!u(\lam, t)\,dt=\log\left(1+
\frac{\lam}{2\theta}\right).
\end{equation}

Let      $ c(t)= \lim_{\lam \to \infty} u(\lam, t)$ and denote by
$\zeta=\inf\{ t>0;  Y_t=0\}$ the lifetime of $Y$ under $\N$. Then we
have for $t>0$:
 \begin{equation}\label{eq:defc}
 c(t) =\N[\zeta > t] = \frac{2 \theta}{\ebt - 1}\cdot
 \end{equation}
From the Markov property of $Y$, we deduce that for $s>0$ and $t,\lambda\geq 0$:
\begin{equation}
   \label{eq:uc=c}
u(u(\lambda,s),t)=u(\lambda, t+s) \quad \text{and}\quad u(c(s), t)=c(t+s).
  \end{equation}
 We deduce from \reff{eq:int-u} that for $s>t>0$:
\begin{equation}
   \label{eq:int-c}
\beta\int_s^{+\infty} c(r)\, dr=\beta \int_0^{+\infty } u(c(s), r) \, dr= \log
\left(1+ \frac{ c(s)}{2\theta} \right)= -\log\left(1-\expp{-2\beta
    \theta s} \right)
\end{equation}
as well as
\begin{equation}
   \label{eq:intc}
\beta\int_t^s c(r)\, dr=  \log
\left(\frac{1-\expp{-2\beta\theta s}}{1- \expp{-2\beta\theta t}}\right).
\end{equation}
We easily get the following results for $t>0$:
\begin{equation}\label{eq:nyt}
\N [Y_t] = \embt
\end{equation}
as well as
\begin{equation}
\label{eq:NeYzeta}
\N[\expp{-\lambda Y_t} \ind_{\{\zeta>t\}}]
= c(t) -u(\lambda,t),
\end{equation}
 and, thanks to the Markov property of $Y$ and \eqref{eq:backe} for
 $s>0$, $t>0$:
\begin{equation}\label{eq:yrys0}
\N[Y_{s}\ind_{\{Y_{s+t}=0\}}]
=\N\left[Y_{s}\expp{-c(t)Y_{s}}\right]
=\frac{\psi(c(s+t))}{\psi(c(t))}
=\eb{s} \left(\frac{c(s+t)}{c(t)}\right)^2\cdot
\end{equation}

\subsection{Genealogy of the CB process $Y$}
\label{sec:genY}
We will recall the genealogical tree for the CB process which is studied
in \textsc{Le  Gall} \cite{lg:iert} or  \textsc{Duquesne} and \textsc{Le
  Gall} \cite{dlg:rtlpsbp}. Since  the branching mechanism is quadratic,
the corresponding L\'evy process is just the Brownian motion with drift.
Let  $ W  = (W_t,  t \in  \R_+)  $ be  a standard  Brownian motion.   We
consider  the Brownian  motion $W^\theta=(W_t^\theta,  t\in  \R_+)$ with
negative drift and the corresponding reflected process above its minimum
$H=(H(t), t\in \R_+)$:
$$
W_t^\theta=\sqrt{\frac{2}{\beta}} W_t - 2\theta t
\quad\text{and}\quad
H(t)= W_t^\theta -\inf_{s\in [0,t]} W_s^\theta.
$$
We deduce from equation $(1.7)$ in \cite{dlg:rtlpsbp} that $H$ is the
height process associated to the branching  mechanism $\psi$. For a
function $h$ defined on $\R_+$, we set:
$$
\max(h)=\max(h(t), t\in \R_+).
$$
Let $\rN[dH]$ be the excursion measure of $H$ above $0$ normalized  such that $\rN[\max(H)\geq r]= c(r)$. Let  $(\ell_t^x(H), t\in \R_+, x\in \R_+)$ be the local time of    $H$ at time $t$ and level $x$. Let $\zeta=\inf\{t>0; H(t)=0\}$   be the duration of the excursion $H$ under $\rN[dH]$.  We recall that $(\ell_{\zeta}^r (H), r\in\R_+)$ under $\rN$ is distributed as $Y$ under $\N$.       From now on we shall identify $Y$ with    $(\ell_{\zeta}^r(H), r\in \R_+)$ and write $\N$ for $\rN$. We now recall the construction of the genealogical tree of the CB process  $Y$ from $H$.

Let      $f$ be a  continuous  non-negative function defined  on $[0,+\infty)$, such that $f(0)=0$,  with compact support. We set $\zeta^f = \sup\{t; f(t) > 0\}$,        with the convention that $\sup\emptyset=0$.        Let $d^f$ be the non-negative function
defined by:
$$
d^f(s,t)=f(s)+f(t)- 2 \inf_{u\in [s\wedge t, s\vee t]} f(u).
$$
It can be easily checked that    $d^f$ is a semi-metric on   $[0, \zeta_f]$. One can define the  equivalence relation associated  to $d^f$ by $s\sim t$ if and only if  $d^f(s,t)=0$. Moreover, when we consider the quotient space $T^f=[0,\zeta_f]/_{\sim}$ and,  noting again  $d^f$ the induced metric on  $T^f$ and rooting     $T^f$ at $\emptyset^f$,     the equivalence class of $0$, it can be checked that the space    $(T^f,d^f,\emptyset^f)$ is a compact rooted real tree.

The so-called  genealogical tree of the CB process $Y$ is the real tree
$\Tau=(T^H, d^H, \emptyset^H)$. In what follows,     we shall mainly
present the result using the height process  $H$ instead of the
genealogical tree $\Tau$, and say that $H$ codes for the genealogy of
$Y$.

Let $a>0$ and  $(H_k, k\in \ck_a)$ be the excursions  of $H$ above level
$a$.  It  is well  known that $\sum_{k  \in \ck_a}  \delta_{H_k}(dH)$ is
under $\N$  and conditionally on $(Y_r,  r \in [0,a])$,  a Poisson point
measure with intensity  $Y_a \N[dH]$. We define the  number of ancestors
at  time $a$  of the  population living  at time  $b$ as  the  number of
excursions above level $a$ which reach level $b > a$ by:
$$
R_a^b(H) = \sum_{k \in \ck_a} \ind_{\{\max(H_k) \geq b - a\}}.
$$
When there is no confusion, we shall write $R_a^b$ for $R_a^b(H)$.
Notice that $R_a^b$ is conditionally on $Y_a$ a Poisson random variable with mean
$c(b-a) Y_a$.

We compute functionals of $R$ in Section \ref{sec:app-R}.

\subsection{The population model}
\label{sec:pop-mod}
We model the population using a stationary CB process.
Let $\D$ be the space of c\`adl\`ag paths having $0$  as a
trap. Consider under $\rP$ a Poisson point measure
\begin{equation}
   \label{eq:defcn}
\cn(dt, dY)=\sum_{i \in I} \delta_{(t_i, Y^i)}(dt, dY)
\end{equation}
     on
$\R \times \D$  with intensity $2 \beta dt \N[dY]$.   We shall consider the process $Z=(Z_t, t\in \R)$ defined by
\begin{equation}\label{defz}
Z_t= \sum_{t_i \leq t} Y_{t-t_i}^i.
\end{equation}
Let  $\rE$ be  the  expectation  with respect  to  $\rP$.  According  to
\cite{cd:spsmrcatsbp},  $Z$ is  a  CB process with  branching mechanism  $\psi$,
conditionally on non-extinction.  Notice the process $Z$ is a.s. finite,
a.s.  positive and  stationary. We shall model a  population with random
size  by the  process $Z$.  The process  $Z$ can  be seen  as a  CB process with
immigration  or  a population  with  an  infinite  lineage (or  immortal
individual).

Using  the property  of  the Poisson  point
measure, we have:
\begin{equation}
   \label{eq:exp-Z}
\rE[\expp{-\lam Z_t}]=\left(1+\frac{\lam}{2 \theta}\right)^{-2},
\end{equation}
which also gives:
\begin{equation}\label{eq:momentz}
\rE[Z_t]=\frac{1}{\theta} \quad \text{and}\quad \rE[Z_t^2]=\frac{3}{2\theta^2}\cdot
\end{equation}
Using the branching property of $Y$, it is easy to get for $s\geq 0$:
\begin{equation}
   \label{eq:EZ0Zs}
\E[Z_0 Z_s ]=\frac{2+\embs}{2\theta^2}\cdot
\end{equation}

\section{The number of ancestors process}\label{sec:ancespro}
\subsection{Definition}\label{sec:AP-def}

We describe the genealogy of $Z$ using the framework developed in
Section \ref{sec:genY}. Let
\[
 \cn'(dt, dH)=\sum_{i \in I}\delta_{(t_i,  H^i)}(dt, dH)
\]
be a  Poisson point  measure with intensity  $2 \beta\, dt  \N(dH)$.  We
will write $Y^i_a$ for $\ell^a(H^i)$  for $i \in I$ and use \eqref{defz}
for  the definition  of  $Z$.   Let $\Tau^i$  be  the genealogical  tree
associated to  $H^i$. Consider the real  line as an  infinite spine, and
for  all  $i\in I$,  graft  the  tree $\Tau^i$  at  level  $t_i$ on  the
infinite spine.   This defines  a tree which  we call  the genealogical
tree of  the process  $Z$.  Thus $\sum_{i  \in I} \delta_{(t_i,  H^i )}$
allows to  code (on an enlarged  space) the genealogy of  $Z$ defined by
\eqref{defz}.

Let $r < t$. We define the number of ancestors, excluding the immortal
individual, at time  $r$ of the  population living at time $t$, $M_r^t$, by:
\begin{equation}\label{def:mrt}
M_r^t =\sum_{i\in I}\ind_{\{t_i < r\}} R_{r - t_i}^{t-t_i}(H^{i}).
\end{equation}
We  shall identify $M_{-r}$  with $M_{-r}^0$ for $r>0$,  when there  is no  risk of
confusion.  The time  to the most recent common  ancestor (TMRCA) of the
population living at time $0$ is defined as $\inf\{r>0; M_{-r}=0\}$.  We
shall  call $(M_r^t,  -\infty <r<t<+\infty  )$ the  number  of ancestors
process.

\begin{rem}
   \label{rem:planar}
   Notice the time order on $H^i$ allows to define an order structure on
   $\Tau^i$,  which could  then be  described  as a  planar tree.   Then
   grafting $\Tau^i$ at  level $t_i$ either on the left  or on the right
   of the infinite spine would  define a planar genealogical tree of the
   process $Z$.  Since this order structure is of no use to the  study of the
   length of the genealogical tree, we decide to
   omit it and concentrate on the number of ancestors process instead.
\end{rem}

Recall  from  \cite[Section 6]{cd:spsmrcatsbp},  that  conditionally  on
$(Z_{-u}, u\geq r) $, the random variable $ M_{-r} $ is a Poisson random
variable    with   intensity    $c(r)Z_{-r}$.   This    implies,   using
\reff{eq:momentz} and \reff{eq:exp-Z} that for $t>0$:
\begin{equation}
   \label{eq:EM}
   \rE[M_{-t}]=\frac{c(t)}{\theta},\quad
   \rE\left[M_{-t}^2 \right]=\frac{c(t)}{\theta}\left(1+
     \frac{3}{2}\frac{c(t)}{\theta}\right)= 2\frac{\expp{2\beta\theta
       t}+2}{(\expp{2\beta\theta t} -1)^2},
\end{equation}
\begin{equation}
\label{eq:lapm}
\rE[\expp{-\lambda M_{-r}}]
=\rE[\expp{-(1-\expp{-\lambda})c(r) Z_{-r}}]
= \left(1+\frac{c(r)}{2\theta} (1- \expp{-\lambda})\right)^{-2},
\end{equation}
and moreover a.s.:
\begin{equation}
   \label{eq:limM/c}
\lim_{r\to 0+}\frac{M_{-r}}{c(r)} =Z_0.
\end{equation}

\subsection{Associated birth and death process}
\label{sec:AP-bd}

Thanks to the  branching property, we get that  the process $(M_t, t<0)$
is a birth process starting from 0  at $-\infty $. The birth rate is the
sum of  two terms:  the first  one is the  contribution of  the immortal
individual and it is  equal to $2 \beta c(-t) \, dt$;  the second one is
the contribution of the current  ancestors and is equal to $ \beta
c(-t)  M_{t-}\, dt$, see Proposition \ref{prop:Y|R}. We deduce the
following result.

\begin{prop}
   \label{prop:M-mart}
The process $(M_t, t<0)$
is a càdlàg birth process starting from 0  at $-\infty $ with rate $\beta c(|t|)
\left(M_t+2\right) $ at time $t<0$. Equivalently, the process $(\tilde M_t,
t<0)$ defined by $\lim_{t\rightarrow -\infty } \tilde M_t=0$ and
\[
d\tilde M_t = dM_t - \beta c(|t|) \left(M_t+2\right) dt
\]
is a martingale (with respect to its natural filtration) whose jumps are
equal to 1.
\end{prop}

Similarly,  we can check the following result.
\begin{prop}
   \label{prop:M-back}
The process   $(M_{(-t)_-}, t>0)$
is a càdlàg death process with rate:
\[
M_{-t} |c'(t)|/c(t)=\beta
M_{-t}(2\theta+c(t)).
\]
 \end{prop}

We can also recover the process $(Z_t, t<0)$ from $(M_t, t<0)$ by
grafting CB processes on the number of ancestors process. Notice there is a
contribution from the immortal individual with rate $2\beta
\N\left[dY; \zeta<|t|\right]\, dt$ (as we do not take into account the
contributions which reach the current time $0$) and from the
genealogical tree, according to Proposition \ref{prop:Y|R}, we have the
contributions of $ Y^{\text{(g)},i}$ and we only keep the
contributions of $ Y^{\text{(d)},i}$ which do not reach the current time
$0$; this gives a contribution with rate   $2\beta M_t
\N\left[dY; \zeta<|t|\right]\, dt$. Therefore, we have the following result.

\begin{prop}
   \label{prop:ZfromM}
Let $\sum_{i \in I} \delta_{t_i, \tilde Y_i}(dt, dY)$  be, conditionally on
$(M_t, t<0)$, a Poisson point measure  on
$(-\infty ,0) \times \D$ with intensity:
\[
2\beta (M_t+1)  \N\left[dY; \zeta<|t|\right]\, dt.
\]
Then, conditionally on
$(M_t, t<0)$, the process $(\tilde Z_t, t<0)$ is distributed as $(Z_t, t<0)$ where
for all $t<0$:
\[
\tilde Z_t= \sum_{t_i \leq t} \tilde Y_{t-t_i}^i.
\]
\end{prop}
Moments for the process $(M_t, t<0)$ are given in Section
\ref{sec:MomZ}.

We deduce from Proposition \ref{prop:ZfromM}, the following remarkable formula on the weighted integral of
the number of ancestors process.

\begin{cor}
   \label{cor:exp-M}
Let $t>0$. We have:
\[
\rE\left[\expp{- 2\beta \int _t^{+\infty } dr\, \big(c(r-t)-
      c(r)\big)M_{-r} }\right]
= \left(\frac{2\theta}{2\theta+c(t)}\right)^2=\rE\left[\expp{-c(t)
    Z_0}\right].
\]
\end{cor}

\begin{proof}
   According to Proposition \ref{prop:ZfromM}, we have:
\[
\rE\left[\expp{-\lambda Z_t} \right]
=\rE\left[\exp\left(- 2\beta \int _t^{+\infty }dr\,  (M_{-r}+1)
\N\left[(1-\expp{-\lambda Y_{r-t}})\ind_{\{\zeta<r\}}\right]
  \right)\right].
\]
Notice that:
\[
\N\left[(1-\expp{-\lambda Y_{r-t}})\ind_{\{\zeta<r\}}\right]=
u(\lambda+c(t), r-t) - u(c(t), r-t)=u(\lambda+c(t), r-t) - c(r).
\]
Thanks to \reff{eq:int-u} and \reff{eq:int-c}, we get:
\[
2\beta \int _t^{+\infty }dr\, \big(u(\lambda+c(t), r-t) - c(r)\big)
= 2 \log \left(1+ \frac{\lambda + c(t)}{2\theta} \right) - 2\log
\left(1+ \frac{ c(t)}{2\theta} \right) .
\]
Then use \reff{eq:exp-Z} to get:
\[
\rE\left[\exp\left(- 2\beta \int _t^{+\infty }dr\,\big(u(\lambda+c(t),
    r-t) - c(r)\big) M_{-r}  \right)\right]
= \left(\frac{2\theta}{2\theta +\lambda}\right)^2 \left(\frac{2\theta
      +\lambda+c(t)}{2\theta + c(t)}\right)^2.
\]
Letting $\lambda$ goes to infinity and  \reff{eq:exp-Z} give  the result.
\end{proof}

\section{Time reversal of the number of ancestors process}
\label{sec:timereversal}
The next result is in a sense a consequence of the time
reversibility of  the process $Y$ with respect to its lifetime
$\zeta$.

\begin{lem}
   \label{lem:MZ}
The random variable $(Z_0, (M_{-t}, t>0))$ and $(Z_0, (M_0^{t}, t>0))$ have the same
distribution.
\end{lem}
This result will be generalized in  Theorem \ref{theo:time-reversal}.

\begin{rem}
   \label{rem:AldPop}
Up to a random labeling of the individuals, see Remark \ref{rem:planar}
in our setting, the lineage tree defined in \cite{ap:cbpmb} or \cite{p:agcbp}
of the
population living at time 0 is given by the coalescent times of the
genealogical tree  or equivalently by the jumping times of the process
$(M_t, <>0)$:
\[
\ca=\{|t|; t<0 \text{ s.t. } M_t
-M_{t_-} =1\}.
\]

Let $\sum_{j\in J} \delta_{x_j,\hat Y^j}$  be a Poisson point measure on
$(0,+\infty )\times  \D$ with intensity  $dx \N[dY]$ and  independent of
$Z_0$.  Let  $\hat\zeta_j$  denote   the  lifetime  of  $\hat  Y^j$.  By
considering the genealogies and using the branching property, we get:
\begin{equation}
   \label{eq:M=zeta}
(M_0^{t}, t>0) \stackrel{\text{(d)}}{=} \left(\sum_{x_j<Z_0}
  \ind_{\{\hat \zeta_j\geq t\}}, t>0\right).
\end{equation}
Then, thanks to Lemma \ref{lem:MZ}, we deduce that the coalescent times
$\ca$ are distributed as the family of lifetimes:
\[
\ca \stackrel{\text{(d)}}{=} \{\hat \zeta_j; j\in
J\text{ s.t. } x_j<Z_0\}.
\]
Notice that by construction, $\sum_{j\in J} \delta_{x_j,\hat \zeta_j}$
is  a Poisson point measure on $(0,+\infty )^2$
 with intensity  $dx |c'(t)| dt$ and  independent of
$Z_0$.

This result is similar to the one in \cite{ap:cbpmb} or \cite{p:agcbp}
for a critical CB process (corresponding to $\theta=0$ in our framework)
born in the past according to the Lebesgue measure on $(-\infty , 0)$,
with the intensity of the corresponding Poisson point measure on
$(0,1)\times (0,+\infty )$ given by  $dx\,  t^{-2}  dt$; see also  \cite{g:crp}
for extensions concerning the model developed in \cite{p:agcbp}.
(Notice the two
intensities are similar near  $0$ as
$|c'(t)| \sim_{0+} 1/ (\beta t^2) $.)
Similar results are given for other models, see \cite{lp:cppbt} for
non-quadratic CB process,   and \cite{l:apcpp} for
Crump-Mode-Jagers processes.
\end{rem}

\begin{proof}
Notice that a.s. $Z_0=\lim_{t\to 0+}{M_{-t}}/{c(t)}$. Thanks to
\reff{eq:limM/c} and  $c(t)=\N[\zeta\geq t]$,   we can deduce from
\reff{eq:M=zeta}, using standard results on Poisson point measure,
that a.s. $Z_0=\lim_{t\to 0+}{M_0^{t}}/{c(t)}$. This and the fact that
$(M_{-t}, t>0)$ and $(M_0^{t}, t>0)$ are Markov processes, imply that it is enough
to check
that $(M_{-t}, M_{-r})$ and $(M_0^t, M_0^r)$ have the same
distribution for $r>t>0$ to prove the Lemma.

Let $r>t>0$.
On one hand, notice that each of the  $ M_0^t$  ancestors at time $0$ of
the population living at time $t$ generate independently a population
(at time 0) distributed according to $\N[dY|\zeta>t]$. This implies that
\begin{equation}
   \label{eq:repMr}
M_0^r\stackrel{\text{(d)}}{=}\sum_{i=1}^ {M_0^t} \ind_{\{\tilde
    \zeta_i>r\}},
\end{equation}
where  $(\tilde Y^i, i\in \N^*)$ are independent, independent of $M_0^t$
and distributed according to $\N[dY|\zeta>t]$. This readily implies that
$M_0^r$  is,   conditionally  on  $M_0^t$,  binomial  with
parameter $\left(M_0^t,\frac{c(r)}{c(t)}\right)$.  (This could have been
deduced from Corollary \ref{cor:binR}.)
Thus using
\reff{eq:lapm}, we have for $\lambda>0$ and $\mu>0$:
\begin{align}
\nonumber
\rE\left[\expp{-\lambda M_0^r -\mu M_0^t}\right]
&=\rE\left[\left(\expp{-\mu} \left(1 - \frac{c(r)}{c(t)}
      (1-\expp{-\lambda}) \right) \right)^ {M_0^t}\right]\\
\nonumber
&= \left(1+\frac{c(t)}{2\theta}\left(1-\expp{-\mu}\
\left(1-\frac{c(r)}{c(t)}+\frac{c(r)}{c(t)}\expp{-\lambda}\right)\right)\right)^{-2}\\
\label{eq:LapMrMt}
&=\left(1+\frac{c(t)}{2\theta}(1-\expp{-\mu})
+\frac{c(r)}{2\theta}(1-\expp{-\lambda})\expp{-\mu}\right)^{-2}.
\end{align}

On the other hand, given  $M_{-r}$,  $M_{-t}$  can be decomposed into    two parts:
\begin{equation}
\label{eq:decomm}
 M_{-t}=M^{I[-r, -t]}+\sum_{j=1}^{M_{-r}} \tilde M_{r-t}^{r, j},
\end{equation}
where
\begin{enumerate}
\item[(i)] $M^{I[-r, -t]}$ is the number of ancestors at time $-t$ of
  population living at time $0$  corresponding to a population $Y^i$
  (see definition \reff{eq:defcn})
  with immigration time $t_i$ belonging to $(-r, -t)$ and $M^{I[-r,
    -t]}$ is independent of $M_{-r}$.
\item[(ii)] $(\tilde M_{r-t}^{r, j}, j\in \N^*)$ are independent,
  independent of $M_{-r}$ and each one represents the number of ancestor
  at time $-t$ generated by one of the ancestors at time $-r$. By
  construction, $(\tilde M_{r-t}^{r, j}, j\in \N^*)$ are distributed as
  $R_{r-t}^r$ under $\N[dY|\zeta>r]$.
\end{enumerate}
We get  for $\lambda>0$ and $\mu>0$:
\begin{align}
   \nonumber
\rE\left[\expp{-\lambda M_{-r} -\mu M_{-t}}\right]
&=\rE\left[\expp{-\lambda M_{-r}}\rE\left[\expp{-\mu M_{-t}}|M_{-r}\right]\right]\\
\label{expMrMt}
&=\rE\left[\expp{-\lambda M_{-r}} \N\left[\expp{-\mu R_{r-t}^r}|
    \zeta>r\right]^{M_{-r}}\right]\E\left[\expp{-\mu M^{I[-r,-t]}}\right],
\end{align}
Using \reff{eq:nexpR}, we obtain:
\[
 \N\left[\expp{-\mu R_{r-t}^r}| \zeta>r\right]=
\frac{\N\left[\zeta>r\right] - \N\left[1- \expp{- \mu
        R_{r-t}^r}\right]}{\N\left[\zeta>r\right]}=
\frac{c(r)-u((1-\expp{-\mu})c(t), r-t)}{c(r)}
\]
and:
\[
\rE\left[\expp{-\mu M^{I[-r,-t]}}\right]
=\expp{-2\beta \int_0^{r-t} ds \, \N\left[1-\expp{-\mu R_{s}^{t+s}
      }\right]}=\left(1+(1-\expp{-\mu})(1-\expp{-2\beta \theta
    (r-t)})\frac{c(t)}{2\theta}\right)^{-2}.
\]
Plugging the above computations in \reff{expMrMt}, and
 using \reff{eq:lapm}, we get:
\[
\rE\left[\expp{-\lambda M_{-r} -\mu M_{-t}}\right]
=\left[1+\frac{c(t)}{2\theta}(1-\expp{-\mu})
+\frac{c(r)}{2\theta}(1-\expp{-\lambda})\expp{-\mu}\right]^{-2}.
\]

This and \reff{eq:LapMrMt}
imply that $(M_{-t}, M_{-r})$ and $(M_0^t, M_0^r)$ have the same
distribution.
\end{proof}

We now give the main Theorem of this Section on the time reversal of the
number of ancestors process.

\begin{theo}
   \label{theo:time-reversal}
The process  $(M_{s-r}^{s}, s\in \R, r>0)$ is distributed as
$(M_{s}^{s+r}, s\in \R, r>0)$.
\end{theo}

\begin{rem}
   \label{rem:tr-totl}
Using stationarity, we deduce that the truncated total  length  process
of the genealogical
tree, see \reff{defL}:
\[
\left(L_\varepsilon^s= \int_\varepsilon^\infty M_{s-r}^s\, dr, s\in \R,
  \varepsilon>0 \right)
\]
is distributed as the (backward) truncated total lifetime process of the population:
\[
\left(\int_\varepsilon^\infty M_{-s}^{-s+r}\, dr, s\in \R, \varepsilon>0
\right).
\]
In particular, we deduce from Remark \ref{rem:AldPop}  the following
distribution equality:
\begin{equation}
   \label{eq:=timer}
(L_\varepsilon^0, \varepsilon>0) \stackrel{\text{(d)}}{=}
\left(   \sum_{x_j<Z_0} ( \zeta_j -\varepsilon)_+,
  \varepsilon>0\right),
\end{equation}
where $x_+=\max(x,0)$ and $\sum_{j\in J} \delta_{x_j, \zeta_j}$
is  a Poisson point measure on $(0,+\infty )^2$
 with intensity  $dx |c'(t)| dt$  independent of
$Z_0$.
\end{rem}

Before giving the proof of Theorem \ref{theo:time-reversal} which is
postponed at the end of this Section, we first
give a preliminary Lemma.

We define the forward process for the individuals living at time $s$ (which relies on
their life-time) $\cm^{(\text{f})}_s=(M_s^{r+s}, r>0)$ and the backward
process for the ancestors of the population living at time $s$,
$\cm^{(\text{b})}_s=(M_{s-r}^s, r>0)$.

 \begin{lem}
\label{lem:mf=mb}
 We have for $t>r>0$, $s\geq 0$, $\lambda>0$ and $\mu>0$:
 \begin{align}
 \label{eq:MbLap-second}
\E\left[\expp{-\lambda M_{-s-r}^{-s} -\mu M_{-s-t}^{-s} }\big|
  \cm^{(\text{b})}_0\right]
& =  \ck_{\lambda,\mu} \left(\cm^{(\text{b})}_0\right),\\
\label{eq:MfLap-second}
 \E\left[\expp{-\lambda M^{s+r}_ s- \mu M^{s+t}_s}| \cm^{(\text{f})}_0\right]
& = \ck_{\lambda,\mu} \left(\cm^{(\text{f})}_0\right),
\end{align}
with $\ck_{\lambda,\mu} $ some measurable deterministic function
depending on $\lambda$ and
$\mu$.
 \end{lem}

\begin{proof}
  We    first    prove    \reff{eq:MbLap-second}.   Using    Proposition
  \ref{prop:ZfromM},     the     distribution    of     $(M_{-s-r}^{-s},
  M_{-s-t}^{-s})$,  conditionally  on  $\cm^{(\text{b})}_0$,  consists  of
  three parts:
\begin{itemize}
   \item[(i)] The ancestors at time $-s-t$ of the current population at
  $0$,  that is $M_{-s-t}$.
\item[(ii)] The  ancestors coming    from the  immortal
  individual  (or from  the  immigration) over   $v\in(-\infty, -s-t)$
  whose   intensity  is  $2\beta   dv(M_{v}+1)\N[dY; \zeta<-v]$.
\item[(iii)]  The  ancestors coming   from the  immortal
  individual  (or from  the  immigration)  over   $v\in(-s-t,  -s-r)$  whose   intensity  is  $2\beta
  du(M_{v}+1)\N[dY;  \zeta<-v]$.  Notice that  in  this  case, we  have
  $M_{-s-t}^{-s}=0$.
\end{itemize}
This implies:
\begin{multline*}
\E\left[\expp{-\lambda M_{-s-r}^{-s} -\mu M_{-s-t}^{-s} }\big|
  \cm^{(\text{b})}_0\right]\\
=\expp{-(\mu+\lambda)M_{-s-t}}\exp\left(-2\beta
  \int^{s+t}_{s+r}dv\; (M_{-v}+1) \N\left[\left(1-\expp{-\lambda
        R_{v-s-r}^{v-s}}\right)\ind_{\{\zeta<v\}}\right]\right)\\
\exp\left(-2\beta \int^{\infty}_{s+t}dv\; (M_{-v}+1)
  \N\left[\left(1-\expp{-\lambda R_{v-s-r}^{v-s}-\mu
        R_{v-s-t}^{v-s}}\right)\ind_{\{\zeta<v\}}\right]\right).
\end{multline*}
Lemma \ref{lem:Rab} implies
\begin{equation}\label{eq:La_MZeta}
\N\left[\left(1-\expp{-\lambda R^{v-s}_{v-s-r}}\right)\ind_{\{\zeta<v\}}\right]
=u(\delta_1, v-s-r)-c(v),
\end{equation}
 \begin{equation}\label{thirdterm}
 \N\left[\left(1-\expp{-\lambda R_{v-s-r}^{v-s}-\mu
       R_{v-s-t}^{v-s}}\right)\ind_{\{\zeta<v\}}\right]=u(\delta_3,
 v-s-t)-c(v),
\end{equation}
with
\begin{equation}
   \label{eq:def-l}
\delta_1=\left(1-\expp{-\lambda} \right)c(r)+\expp{-\lambda}c(r+s)
\quad\text{and}\quad
\delta_3=(1-\expp{-\mu})c(t)+\expp{-\mu} u(\delta_1, t-r).
\end{equation}
Thus, we deduce \reff{eq:MbLap-second} with:
\begin{equation}
   \label{eq:df-K}
\ck_{\lambda,\mu} \left(\cm^{(\text{b})}_0\right) = \expp{-(\lambda+\mu)
  M_{-s-t} -2\beta \int_{s+r}^{+\infty }  g(v)(M_{-v}+1)\, dv},
\end{equation}
with $g$ defined by:
\begin{equation}
   \label{eq:def-g}
g(v)= \ind_{(s+r,s+t)}(v) u(\delta_1, v-s-r)+\ind_{(s+t, +\infty )}(v)
u(\delta_3, v-s-t)  -c(v) .
\end{equation}

On the  other hand, using  the Williams' decomposition given  in Abraham
and  Delmas  \cite{ad:wdlcrtseppnm},  the distribution  of  $(M^{s+r}_s,
M^{s+t}_s)$,  conditionally  on  $\cm^{(\text{f})}_0$,  consists  of  three
parts:
\begin{itemize}
   \item[(i)] The individuals of  the current populations which have descendants
at  time $s+t$,  that is,  $M_0^{s+t}$.
\item[(ii)] The part coming from the immigration over the
time interval  $[0,s]$ which arrives with  rate
$2\beta dv\N[dY] $.
\item[(iii)] The individuals $i\in I$, for some index set $I$, living at
  time  0  with   life-time  $\zeta_i>s+r$  generate  individuals,  with
  intensity $2\beta dv \N[dY; \zeta<\zeta_i -v]$, over the time interval
  $[0,s]$ and some of them are still alive at time $s+r$ and $s+t$.
 \end{itemize}
Notice the set $\{\zeta_i, i\in I\}=\{t; M_0^t =M_0^{t-}+1\}$ is
measurable  with respect to the $\sigma$-field generated by
$\cm^{\text{(f)}}_0$. Therefore, we have:
\begin{multline*}
\E\left[\expp{-\lambda M^{s+r}_ s- \mu M^{s+t}_s}|
  \cm^{(\text{f})}_0\right]\\
=\expp{-(\mu+\lambda) M_{0}^{s+t}}
 \exp\left(-2\beta \int_{0}^{s}dv\;
  \N\left[\left(1-\expp{-\lambda R_{s-v}^{s+r-v}-\mu
        R_{s-v}^{s+t-v}}\right)\right]\right) \\
 \exp\left(-2\beta \sum_{\zeta_i>s+r}\int_{0}^{s}dv\;
  \N\left[\left(1-\expp{-\lambda R_{s-v}^{s+r-v}-\mu
        R_{s-v}^{s+t-v}}\right)\ind_{\{\zeta<\zeta_i-v\}}\right]\right).
\end{multline*}
We set for $q\geq s+r$:
\[
G(q)= \int_{0}^{s}dv\;  \N\left[\left(1-\expp{-\lambda R_{s-v}^{s+r-v}-\mu
        R_{s-v}^{s+t-v}}\right)\ind_{\{\zeta<q-v\}}\right].
\]
Notice that $G(s+r)=0$.
With this notation we can write:
\begin{equation}
   \label{eq:fM}
\E\left[\expp{-\lambda M^{s+r}_ s- \mu M^{s+t}_s}|
  \cm^{(\text{f})}_0\right]
=\expp{-(\mu+\lambda) M_{0}^{s+t} -2\beta G(+\infty ) - 2\beta \sum_{\zeta_i>s+r} G(\zeta_i)}.
\end{equation}

First, we compute the derivative of $G$ on $(s+r,s+t)$. Thanks to Lemma
\ref{lem:Rab}, see \reff{eq:nrabz1}, we have
for $q\in [s+r,s+t]$:
\[
G(q)
= \int_{0}^{s}dv\;  \N\left[\left(1-\expp{-\lambda
      R_{s-v}^{s+r-v}}\right)\ind_{\{\zeta<q-v\}}\right]
= \int_{0}^{s}dv\; \Big(
 u\left(\gamma_1(q) ,s-v\right)-c(q-v)\Big)
\]
with
\[
\gamma_1(q)=(1-\expp{-\lambda})c(r)+\expp{-\lambda}c(q-s).
\]
Notice that for $q\in (s+r,s+t)$:
\begin{align*}
   \partial_q \int_{0}^{s}dv\;
 u\left(\gamma_1(q) ,s-v\right)
&   =\partial_q \gamma_1(q) \int_{0}^{s}dv\;
\partial_\lambda
u\left(\gamma_1(q) ,s-v\right) \\
&   = - \partial_q \gamma_1(q) \int_{0}^{s}dv\;
\frac{\partial_t
  u\left(\gamma_1(q),s-v\right)}
{\psi(\gamma_1(q))} \\
&   = \partial_q \gamma_1(q)
\frac{\gamma_1(q) - u(\gamma_1(q),s)}
{\psi(\gamma_1(q))}, \\
\end{align*}
where we used \reff{eq:backe} for the second equality.
We also have:
\[
 - \partial_q \int_{0}^{s}dv\; c(q-v)=c(q-s) -c(q).
\]
Recall $\delta_1$ defined in \reff{eq:def-l}.
Elementary  computations yield for $q\in (s+r,s+t)$:
\begin{align*}
c(q)+\partial_q G(q)
&= \partial_q \gamma_1(q)
\frac{\gamma_1(q) - u(\gamma_1(q),s)}
{\psi(\gamma_1(q))} + c(q-s)\\
&=2\theta \frac{(\expp{2\beta \theta(r+s)} - 1) - \expp{-\lambda} \expp{2\beta \theta
    r}( \expp{2\beta \theta s} - 1 )}
{(\expp{2\beta \theta(r+s)} - 1)(\expp{2\beta \theta(q-s)} - 1) -
  \expp{-\lambda} (\expp{2\beta \theta s} - 1)(\expp{2\beta \theta(q-s)}
  - \expp{2\beta \theta r})}\\
   &= u(\delta_1, q-s-r).
\end{align*}
Thanks to \reff{eq:def-g}, we deduce that for $q\in (s+r,s+t)$,
$\partial_q G(q)=g(q)$.

Second, we compute the derivative of $G$ on $(s+t,+\infty )$. Thanks to
Lemma \ref{lem:Rab}, see \reff{eq:nrabz1.5}, we have for $q>s+t$:
\[
G(q)
= \int_{0}^{s}dv\;   \Big(
 u\left(\gamma_2(q) ,s-v\right)-c(q-v)\Big),
\]
with
\[
\gamma_2(q)=(1-\expp{-\lambda})c(r)+\expp{-\lambda}(1-\expp{-\mu}) c(t)
+ \expp{-(\lambda+\mu)} c(q-s).
\]
Similar arguments as in the first part give  for $q>s+t$:
\[
c(q)+\partial_q G(q)
= \partial_q \gamma_2(q)
\frac{\gamma_2(q) - u(\gamma_2(q),s)}
{\psi(\gamma_2(q))} + c(q-s).
\]
Recall $\delta_3$ defined in \reff{eq:def-l}. Elementary (but
tedious) computations yield  for $q>s+t$:
\[
c(q)+\partial_q G(q)
 = u(\delta_3, q-s-t)
\]
so that for $q>s+t$,
$\partial_q G(q)=g(q)$.

For all $0<v<s$, $t>r>0$, we have $\N$-a.e. $\lim_{q\downarrow
  s+t} R^{s+t-v}_{s-v} \ind_{\{\zeta<q-v\}} =0$. This implies  that $G$ is
continuous at $s+t$. Since $G(s+r)=0$, we deduce that for all $q\geq
s+r$,
\[
G(q)=\int_{s+r}^q g(v)\, dv.
\]
In particular, we get $G(+\infty )= \int_{s+r}^{+\infty } g(v)\,dv$ as
well as:
\[
\sum_{i; \zeta_i>s+r} G(\zeta_i)
=\int_{s+r}^{+\infty } dv\, g(v) \sum_{i; \zeta_i>s+r} \ind_{\{v\leq \zeta_i\}}
= \int_{s+r}^{+\infty } dv\, g(v) M_0^v.
\]
This, \reff{eq:fM} and \reff{eq:df-K} imply:
\[
\E\left[\expp{-\lambda M^{s+r}_ s- \mu M^{s+t}_s}|
  \cm^{(\text{f})}_0\right]
=\expp{-(\mu+\lambda) M_{0}^{s+t} -2\beta  \int_{s+r}^{+\infty }
dv\, g(v) \left( M_0^v+1\right)}
=\ck_{\lambda,\mu} \left(\cm^{(\text{f})}_0\right) .
\]
This ends the proof of the Lemma.
\end{proof}

\begin{proof}[Proof of Theorem \ref{theo:time-reversal}]
Notice the (stationary) processes $\cm^{(\text{b})}=(\cm^{(\text{b})}_s, s\in \R)$
and $\cm^{(\text{f})}=(\cm^{(\text{f})}_s, s\in \R)$ are Markov.
Since the process  $M_{-s-\bullet}^{-s}$ (resp. $M^{s+
  \bullet}_s$) conditionally on
$\cm^{(\text{b})}_0$ (resp. $\cm^{(\text{f})}_0$)  is Markov, we deduce
from Lemma \ref{lem:mf=mb} that the transition kernel of
$\cm^{(\text{b})}$ and $\cm^{(\text{f})}$ are equal. Then use
Lemma \ref{lem:MZ} to conclude.
\end{proof}

\section{The total length process}\label{sec3:tlp}

\subsection{Total length process}

We define the total length of the genealogical tree for the population
living at time $s$, up to time $s-\varepsilon$ (with $\varepsilon>0$) by:
\begin{equation}
\label{defL}
L_\varepsilon^s= \int_\varepsilon^\infty M_{s-r}^s\, dr.
\end{equation}
In order to study the asymptotic of $L_\varepsilon^s$ as $\varepsilon$ goes to $0$, we consider the normalized total length
\begin{equation}
\label{eq:defcl}
\cl_\varepsilon^s=L_\varepsilon^s-Z_s\int_\varepsilon^\infty c(r)\, dr
= L_\varepsilon^s+\frac{Z_s}{\beta} \log (1-\expp{-2\beta\theta
  \varepsilon}),
\end{equation}
where we used \reff{eq:int-c} for the last equality.
When $s=0$, we write $L_\varepsilon$ (resp. $\cl_\varepsilon$) for
$L_\varepsilon^s$
(resp. $\cl^s_\varepsilon$). By stationarity, the distributions of
$(L_\varepsilon^s, \varepsilon>0)$ and $(\cl_\varepsilon^s, \varepsilon>0)$ do not depend on $s$.
We have:
\begin{equation}
   \label{eq:EL}
\rE[L_\varepsilon] = \inv{\theta} \int _\varepsilon^\infty  c(r)\; dr= -\frac{1}{\beta
  \theta} \log (1-\expp{-2\beta\theta \varepsilon})
\quad\text{and}\quad
\rE[\cl_\varepsilon] = 0.
\end{equation}

In order to study the convergence of $\cl_\varepsilon$, we first give an elementary Lemma.
Recall the  dilogarithm function is defined for $0\leq t\leq 1$ by:
\[
\Li_2(t)=- \int_0^t \frac{\log(1-x)}{x} \, dx
\]
and we have: $\Li_2(0)=0$,  $\Li_2(1)=\pi^2/6$.

\begin{lem}
\label{lem:ntlp2}
For $\eta>\varepsilon>0$, we have:
\begin{equation}
   \label{eq:Les}
\rE\left[\left(\cl_\eta-\cl_\varepsilon\right)^2\right]=
\frac{1}{\beta^2\theta^2}\left[\Li_2(1-\expp{-2\beta\theta
    \eta})-\Li_2(1-\expp{-2\beta\theta\varepsilon})\right]
    +\frac{2\varepsilon}{\beta\theta}\log\left(
    \frac{1-\expp{-2\beta\theta\varepsilon}}{1-\expp{-2\beta\theta
    \eta}}\right).
\end{equation}
In particular, we have $\lim_{\varepsilon\rightarrow 0}
\rE\left[(\cl_\varepsilon)^2\right]=\frac{1}{\beta^2\theta^2}\frac{\pi^2}{6}\cdot$
\end{lem}

\begin{proof}
Notice that:
\[
\rE\left[\left(\cl_\eta-\cl_\varepsilon\right)^2\right]=
\rE\left[\left(\int_\varepsilon^\eta (M_{-r}- c(r) Z_0) \, dr
  \right)^2\right].
\]
   We have:
\[
  \rE\left[\left(\int_\varepsilon^\eta (M_{-r}- c(r) Z_0) \, dr
    \right)^2\right]
=  2\int_{[\varepsilon,\eta]^2} dtdr\ind_{\{t<r\}} \rE \left[(M_{-r}- c(r)
  Z_0 )(M_{-t}- c(t) Z_0 )\right].
\]
It is easy to derive that for $r>t>0$:
\begin{align*}
\rE[c(r)c(t) Z_0^2]
&=\frac{3}{2}\frac{c(r)c(t)}{\theta^2}, \\
  \rE \left[M_{-r}M_{-t}\right]
&= \frac{c(r)}{\theta}\left(1+
  \frac{3}{2}\frac{c(t)}{\theta}\right),\\
\rE \left[M_{-r}c(t) Z_0 \right]
&=
\rE \left[M_0^rc(t) Z_0 \right] =c(r)c(t)\rE\left[Z_0^2\right]=
\frac{3}{2}\frac{c(r)c(t)}{\theta^2}, \\
\rE \left[M_{-t}c(r) Z_0 \right]
&=
\frac{3}{2}\frac{c(r)c(t)}{\theta^2},
\end{align*}
where we used \reff{eq:momentz} for the first equality, \reff{eq:MrMt}
for the second, Lemma \ref{lem:MZ} for the third, and the fact that
conditionally on $Z_0$, $M_0^r$ is Poisson with parameter $c(t) Z_0$ for
the fourth.
We deduce  that for $r>t>0$:
\[
\rE \left[(M_{-r}- c(r)
  Z_0 )(M_{-t}- c(t) Z_0 )\right]= \frac{c(r)}{\theta}
\]
and thus:
\begin{multline*}
\rE\left[\left(\int_\varepsilon^\eta (M_{-r}- c(r) Z_0 )\, dr
    \right)^2\right]\\
\begin{aligned}
&=\frac{2}{\theta} \int_\varepsilon^\eta (r-\varepsilon)c(r)\, dr\\
&=-\inv{\beta^2 \theta^2}  \int_{\expp{-2\beta\theta
    \eta}}^{\expp{-2\beta\theta\varepsilon}} \frac{\log(y)}{1-y}\, dy
+\frac{2\varepsilon}{\beta\theta}\int_{\expp{-2\beta\theta
    \eta}}^{\expp{-2\beta\theta\varepsilon}}\frac{1}{y-1}dy\\
    &=\frac{1}{\beta^2\theta^2}\left[\Li_2(1-\expp{-2\beta\theta
    \eta})-\Li_2(1-\expp{-2\beta\theta\varepsilon})\right]
    +\frac{2\varepsilon}{\beta\theta}\log\left(
    \frac{1-\expp{-2\beta\theta\varepsilon}}{1-\expp{-2\beta\theta
    \eta}}\right).
\end{aligned}
\end{multline*}
The second assertion is immediate.
\end{proof}

We have the a.s. and the $L^2$ convergence of $(\cl_\varepsilon^s,
\varepsilon>0)$ as $\varepsilon$ goes down to $0$.

\begin{theo}\label{cl2conver}
Let $s\in \R$. The compensated tree length  $(\cl_\varepsilon^s,
\varepsilon>0)$ converges a.s. and in $L^2$:
there exists a random variable $W_s\in L^2$ such that
\[
\cl_\varepsilon^s \xrightarrow[\varepsilon\to 0]{\text{$L^2$ and a.s.}} W_s.
\]
We have:
\[
\rE[W_s]=0 \quad\text{and}\quad
\rE[W_s^2]=\frac{1}{\beta^2\theta^2}\frac{\pi^2}{6}\cdot
\]
\end{theo}

By  stationarity, we  deduce that  the  distribution of  $W_s$ does  not
depend on $s$.  Furthermore, we have the  convergence a.s. and in $L^2$ of
the  finite dimensional  marginals of  the  process $(\cl_\varepsilon^s,
s\in  \R)$  towards   those  of  the  process  $W=(W_s,   s\in  \R)$  as
$\varepsilon$ goes down to $0$.

\begin{proof}
  By stationarity,  we only need to  consider the case  $s=0$. We deduce
  from Lemma \ref{lem:ntlp2} the $L^2$ convergence of $(\cl_\varepsilon,
  \varepsilon>0)$  as $\varepsilon$  goes down  to $0$  towards  a limit
  $W_0$ as well as the first and second moment of $W_0$.

We now prove the a.s. convergence. We deduce from Lemma \ref{lem:ntlp2}
that for $\eta>0$ small enough,
\[
\rE\left[\left(\cl_\eta-W_0\right)^2\right]=
\frac{1}{\beta^2\theta^2}\Li_2(1-\expp{-2\beta\theta
    \eta}) \leq  \frac{4\eta}{\beta\theta}\cdot
\]
Set $a_n=1/n^2$ for $n\in \N^*$.
We deduce that $(\cl_{a_n}, n\in \N^*)$ converges a.s. to $W_0$. For
$\varepsilon\in [a_{n+1}, a_n]$, we have:
\[
\cl_{a_n} - Z_0\int_{a_{n+1}}^{a_n} c(r)\, dr \leq \cl_\varepsilon \leq  \cl_{a_{n+1}} +
Z_0\int_{a_{n+1}}^{a_n} c(r)\, dr.
\]
Since for $n$ large enough:
\[
\int_{a_{n+1}}^{a_n} c(r)\, dr\leq  c(a_{n+1}) (a_n -a_{n+1}) \leq
\frac{5}{\beta n},
\]
we deduce that $(\cl_\varepsilon,
  \varepsilon>0)$ converges a.s. to $W_0$  as $\varepsilon$  goes down
  to $0$.
\end{proof}

\begin{rem}
   \label{rem:equivL}
We deduce from Lemma
\ref{lem:ntlp2} and \reff{eq:momentz} that:
\[
\rE\left[(L_\varepsilon)^2\right]
= \rE\left[(\cl_\varepsilon)^2\right]+ \frac{2}{\beta}
\log(1-\expp{-2\beta\theta \varepsilon}) \rE[Z_0L_\varepsilon] - \frac{3}{2\beta^2\theta^2}
\log(1-\expp{-2\beta\theta \varepsilon})^2.
\]
Arguing as in the proof of Lemma
\ref{lem:ntlp2} we get:
\[
\rE[Z_0L_\varepsilon]=\int_\varepsilon^\infty c(r)\, dr\, \rE[Z_0^2]=\frac{3}{2\beta\theta^2}
\log(1-\expp{-2\beta\theta \varepsilon}).
\]
This gives:
\[
\rE\left[(L_\varepsilon)^2\right]
=\rE\left[(\cl_\varepsilon)^2\right]+ \frac{3}{2\beta^2\theta^2}
\log(1-\expp{-2\beta\theta \varepsilon})^2.
\]
We get
the following equivalent
for the expectation and variance
of $L_\varepsilon$ as $\varepsilon$ goes down to $0$:
$$
\rE[L_\varepsilon]\sim_{0+} \frac{1}{\beta \theta}
\log (1/ \varepsilon)
\quad\text{and}\quad
\text{Var}(L_\varepsilon)\sim_{0+} \frac{1}{2 \beta^2 \theta^2}
\log(1/\varepsilon)^2.
$$
\end{rem}

\subsection{Distribution and fluctuation for the 1-dimensional marginal}

We  provide the distribution of $W_0$ via its Laplace transform.
\begin{lem}\label{laplace-W}
For $\lambda>0$ and $z>0$, we have:
\[
\rE\left[\expp{-2 \beta\theta\lambda W_0}\,|\,Z_0 =\frac{z}{2\theta}\right]
=\expp{-z\varphi(\lambda)}
\quad\text{and}\quad
\rE\left[\expp{-2\lambda\beta \theta  W_0}\right]
=\left(1+\varphi(\lambda)\right)^{-2},
\]
with
\[
\varphi(\lambda)=  -\lambda \int_0^1 dv\, \frac{1-v^{\lambda}}{1-v}\cdot
\]
\end{lem}
From the proof of below, we
get that the distribution of $W_0$ is infinitely
divisible (conditionally on $Z_0$ or not).
Notice that for $\lambda= n$, we get $\varphi(n)=- n H_n$, where $H_n$ is the
harmonic number.

\begin{proof}
We use notations from Remark \ref{rem:tr-totl}.
According to Remark \ref{rem:tr-totl}, see also \reff{eq:=timer}, and
since  $c(t)=\N[\zeta>t]$, we get that
$(Z_0,(\cl_\varepsilon, \varepsilon>0))$ is distributed as
$(Z_0, (\tilde \cl_\varepsilon, \varepsilon>0))$, with:
\[
\tilde \cl_\varepsilon= \sum_{x_j<Z_0} (\zeta_j -\varepsilon)_+ -Z_0
\N[(\zeta-\varepsilon)_+].
\]
In particular
$W_0$ is distributed as
$\tilde W_0= \lim_{\varepsilon \downarrow  0} \tilde \cl_\varepsilon$.

The exponential formula for Poisson point measure gives for any $\lambda>0$:
\[
\rE\left[\expp{-\lambda \tilde \cl_\varepsilon}|Z_0 =\frac{z}{2\theta}\right]
= \exp \left( -\frac{z}{2\theta} \N\left[1-\expp{-\lambda (\zeta-\varepsilon)_+}\right]
+ \frac{\lambda z}{2\theta} \N[(\zeta-\varepsilon)_+]\right)
= \exp ( -zK_\varepsilon(\lambda))
\]
with
\begin{align*}
   K_\varepsilon(\lambda)
&=\frac{1}{2\theta}\N\left[1- \expp{-\lambda
    (\zeta-\varepsilon)_+} -\lambda   (\zeta-\varepsilon)_+  \right]\\
&= 2\beta\theta\int_\varepsilon^\infty dt \; \frac{\expp{2\beta\theta t}}{(\expp{2\beta\theta t} -1)^2}
\left(1- \expp{-\lambda
    (t-\varepsilon)}- \lambda   (t-\varepsilon) \right)\\
&= 2\beta\theta\expp{2\beta\theta \varepsilon} \int_0^\infty dt \;
\frac{\expp{2\beta\theta  t}}{( \expp{2\beta\theta (t+\varepsilon)} -1)^2}
\left(1-\expp{-\lambda    t}- \lambda t  \right),
\end{align*}
where we used that $\N[d\zeta]_{|\zeta=t}=-c'(t)
\;dt=4\beta\theta^2\frac{\expp{2\beta\theta t}}{(\expp{2\beta\theta t}
  -1)^2}\;  dt$ for the second equality.
Notice that by dominated convergence:
\[
 \lim_{\varepsilon \downarrow
  0} K_\varepsilon(\lambda) = \psi(\lambda)
\]
with
\[
\psi(u)= 2\beta\theta \int_0^\infty dt \; \frac{\expp{2\beta\theta t}}{(\expp{2\beta\theta t} -1)^2}
\left(1- \expp{-\lambda
    t}- \lambda  t \right).
\]
Letting $\varepsilon$ goes down to $0$, we  deduce the Laplace transform
of $W_0$, for $\lambda>0$:
\[
\rE\left[\expp{-2 \beta\theta\lambda W_0}|Z_0 =\frac{z}{2\theta}\right]
=\expp{-z\varphi(\lambda)}
\]
with
\begin{align*}
\varphi(\lambda)=\psi(2\beta\theta\lambda )
&= \int_0^\infty dt \; \frac{\expp{ t}}{(\expp{t} -1)^2}
\left(1- \expp{-\lambda
    t}-\lambda  t \right)\\
&= -2\lambda \int_0^\infty dt \; \frac{1-\expp{-\lambda
    t}}{\expp{ t} -1}\\
&= -\lambda \int_0^1 dv\, \frac{1-v^{\lambda}}{1-v},
\end{align*}
where  we used an  integration by  part in  the third  equality.  Notice
that, conditionally on $Z_0$,  $W_0$ is infinitely divisible with L\'evy
measure $|c'(t)|\, dt$.

We also have:
\begin{align*}
\rE\left[\expp{-2\beta\theta\lambda W_0-\mu Z_0}\right]
&=\rE\left[\expp{-\mu Z_0-2\theta\varphi(\lambda) Z_0}\right]\\
&=\left(1+\varphi(\lambda)+\frac{\mu}{2\theta}\right)^{-2}.
\end{align*}
The result follows by taking $\mu=0$.
\end{proof}

We also give the following result on the fluctuations of $\cl_\varepsilon$.

\begin{prop}
   \label{prop:fluctW}
We have the following convergence in distribution:
\[
(Z_0, \sqrt{\beta /\varepsilon} (\cl_\varepsilon-W_0))
\xrightarrow[\varepsilon\to 0]{\text{(d)}}
(Z_0, \sqrt{2Z_0} G ),
\]
with $G\sim \cn(0,1)$ a standard Gaussian random variable independent of
$Z_0$.
\end{prop}
\begin{proof}
   We keep notations from the proof of Lemma \ref{laplace-W}. Mimicking
   the proof of Lemma \ref{laplace-W}, we get for $\lambda\in \R$,
   $\varepsilon>\eta>0$ :
\[
\rE\left[\expp{i\lambda ( \cl_\varepsilon- \cl_\eta)}|Z_0\right]
= \expp{ -Z_0 f_{\varepsilon,\eta}(\lambda)},
\]
with
\[
f_{\varepsilon,\eta}(\lambda)
=\N\left[1- \expp{i\lambda ((\zeta-\varepsilon)_+ - (\zeta-\eta)_+)}
+i\lambda  ((\zeta-\varepsilon)_+ - (\zeta-\eta)_+)\right].
\]
Notice that $0\leq (x-\eta)_+ -(x-\varepsilon)_+\leq  x \wedge
\varepsilon$ for $x>0$. Since $\N[(\zeta\wedge \varepsilon)^2]$ is
finite, we deduce by dominated convergence that
$\lim_{\eta\rightarrow 0} f_{\varepsilon,
  \eta}(\lambda)=f_\varepsilon(\lambda)$,
with
\begin{align*}
f_\varepsilon(\lambda)
&=
\N\left[1- \expp{-i\lambda (\zeta\wedge\varepsilon)}
-i\lambda  (\zeta\wedge\varepsilon)\right]\\
&=4\beta\theta^2 \int_0^\varepsilon dt\; \frac{\expp{2\beta\theta t}}{(
\expp{2\beta\theta t} -1)^2} ( 1-i\lambda t -\expp{-i\lambda t})
+ \frac{2\theta}{\expp{2\beta\theta \varepsilon} -1} (1-i\lambda
\varepsilon - \expp{-i\lambda \varepsilon}).
\end{align*}
We deduce that:
\[
\rE\left[\expp{i\lambda ( \cl_\varepsilon- W_0)}|Z_0\right]
= \expp{ -Z_0 f_{\varepsilon}(\lambda)}
\]
and thus for $\mu\in \R$:
\[
\rE\left[\exp\left({i\lambda \frac{ \cl_\varepsilon-
      W_0}{\sqrt{\varepsilon}}+ i\mu Z_0}\right)\right]
=\rE\left[ \exp\left({ Z_0(i\mu -
      f_{\varepsilon}(\lambda/\sqrt{\varepsilon})) }\right)\right] .
\]
Dominated convergence yields:
\[
\lim_{\varepsilon\rightarrow 0}
f_{\varepsilon}(\lambda/\sqrt{\varepsilon})
= \frac{\lambda^2}{\beta}\cdot
\]
We deduce that:
\[
\lim_{\varepsilon\rightarrow 0}
\rE\left[\exp\left({i\lambda \frac{ \cl_\varepsilon-
      W_0}{\sqrt{\varepsilon}}+ i\mu Z_0}\right)\right]
=\rE\left[ \expp{ Z_0(i\mu - \frac{\lambda^2}{\beta})
    }\right]
= \rE\left[ \exp\left({ i\lambda \sqrt{\frac{2Z_0}{\beta}} G+
i\mu Z_0}\right)\right].
\]
This gives the result.
\end{proof}

\subsection{Path properties of  the process $W$}

We first give the covariance of the process $W$, whose proof is given in
Section \ref{sec:proof-cov}.

\begin{prop}\label{prop:ewswo}
Let $s\in \R_*$.  We have:
\begin{multline*}
\rE[W_0W_s]=\frac{1}{2\beta^2\theta^2}
\left[\frac{\pi^2}{6}\embs +\ebs \Li_2(\embs)\right]\\
 -2(\ebs-\embs)\int_0^\infty \frac{dr}{\eb{r}-1} \int_s^{s+r} \frac{dq}{\eb{q}-1}\cdot
\end{multline*}
\end{prop}

\begin{rem}
\label{rem:cov}
We deduce that there exists some finite positive constant $C$
such that:
$$
\rE[(W_s-W_0)^2]  \sim_{0+} Cs\,  \log(s)^2. $$  This suggests  that the
process $W$ is not continuous. Indeed, recall definition \reff{eq:defcn}
and notice  the process $L^s_\varepsilon$, for  fixed $\varepsilon$, has
jumps at  least at  any time  $\zeta_i-t_i$ for any  $i\in I$  such that
$\zeta_i$, the death time of $Y^i$, is larger than $\varepsilon$. The
same holds for $W$.
\end{rem}

We have the following
result on the existence of a càdlàg version of $W$, whose proof is given
in Section \ref{sec:proofcadlag}.

\begin{theo}
   \label{theo:cadlag}
There exists a c\`adl\`ag $\R$-valued process $W'=(W'_s, s\in \R)$ having the
same finite dimensional marginals as $W$.
\end{theo}

\subsection{Proof of Proposition \ref{prop:ewswo}}
\label{sec:proof-cov}
By Theorem \ref{cl2conver}, we have:
$$
\rE[W_0W_s]=\lim_{\varepsilon\to 0+} \rE[\cl_\varepsilon\cl^s_\varepsilon].
$$
We turn to the calculation of $\rE[\cl_\varepsilon\cl^s_\varepsilon]$
with $\varepsilon$  small enough. We have:
\[
\rE[\cl_\varepsilon\cl^s_\varepsilon]
=\rE\left[\left(L_\varepsilon-Z_0\int_\varepsilon^\infty c(r) \;
    dr\right)\left(L_\varepsilon^s-
Z_s\int_\varepsilon^\infty c(r) \; dr\right)\right]
=B_1-B_2-B_3+B_4,
\]
with $B_1=\rE[L_\varepsilon L_\varepsilon^s]$,
\[
B_2=\rE[L_\varepsilon Z_s]\int_\varepsilon^\infty c(r) \; dr, \quad
B_3=\rE[L_\varepsilon^s Z_0]\int_\varepsilon^\infty c(r) \; dr,
\quad B_4=\left(\int_\varepsilon^\infty c(r) \; dr\right)^2\rE[Z_0 Z_s].
\]

We first compute $B_4$. Using \reff{eq:EZ0Zs} and \reff{eq:int-c}
we get:
\[
B_4=\frac{2+\embs}{2\beta^2\theta^2}\log\left(1- \emb{\varepsilon} \right)^2.
\]

We compute $B_2$. For $-r<0<s$, we have, using Proposition \ref{prop:ZfromM}:
\begin{align*}
\rE[M_{-r}Z_s]&=\rE[M_{-r}\rE[Z_s|\sigma(M_u, u\leq -r)]]\\
&=\rE\left[M_{-r}\left(\sum_{i=1}^{M_{-r}}\tilde Y^i_{s+r}+\sum_{-r<t_i<s}Y^i_{s-t_i}\right)\right]\\
&=\rE[M_{-r}^2]\N[Y_{s+r}|\zeta>r]+\rE[M_{-r}] 2\beta \int_{-r}^s \N[Y_{s-t}]\;dt.
\end{align*}
Since
\[
\N[Y_{s+r}|\zeta>r]=\frac{\N[Y_{s+r}\ind_{\{\zeta>r\}}]}{\N[\zeta>r]}
=\frac{\emb{(s+r)}}{c(r)},
\]
and by \reff{eq:EM} we have:
\[
\rE[M_{-r}Z_s]=\frac{2+\embs}{\theta(\eb{r}-1)}\cdot
\]
Hence, we have:
\[
B_2=\rE[L_\varepsilon Z_s]\int_\varepsilon^\infty c(r) \; dr
=\frac{2+\embs}{2\beta^2\theta^2}\left(\log\left(1- \emb{\varepsilon} \right)\right)^2,
\]
that is, $B_2=B_4$.

For $B_3$, we first compute $\rE[Z_0M_{s-q}^s]$.
By stationary, we have
$\rE[Z_0M_{s-q}^s]=\rE[Z_{-s}M_{-q}]$.
For $\varepsilon<q<s$, using  \reff{eq:EZ0Zs}, we  get:
\[
\rE[Z_{-s}M_{-q}]=c(q)\rE[Z_{-q}Z_{-s}]
=\frac{2+\eb{(r-s)}}{\theta(\eb{r}-1)}\cdot
\]
For $q>s$, a decomposition similar to the one used to compute $B_2$
gives:
\[
\rE[Z_{-s}M_{-q}]=\frac{\eb{s}+2}{\theta(\eb{q}-1)}
-\frac{(\eb {s} - 1)^2}{\theta(\eb{q} - 1)^2}\eb{(q - s)}\cdot
\]
This gives:
\begin{align*}
B_3
&=\int_\varepsilon^\infty c(r) \; dr
\int_\varepsilon^\infty dq\;\rE[Z_0M_{s-q}^s]\\
&=\int_\varepsilon^\infty \frac{c(r)}{\theta}dr\left[\int_\varepsilon^s dq \frac{\eb{(q-s)}+2}{(\eb{q}-1)}  + \int_{s}^\infty dq \left[\frac{\eb{s}+2}{(\eb{q}-1)}
-\frac{(\eb {s} - 1)^2}{(\eb{q} - 1)^2}\eb{(q - s)}\right]\right].
\end{align*}

For $B_1$, the integrand is computed in Lemma \ref{lem:m0ms}.
We have:
\begin{multline*}
B_1=2\int_\varepsilon^\infty dr\int_\varepsilon^{s+r} dq \frac{\eb{(q-s)}+2}{(\eb{r}-1)(\eb{q}-1)}\\
+2\int_\varepsilon^\infty dr\int^\infty_{s+r} dq \left[\frac{\eb{(r+s)}+2}{(\eb{r}-1)(\eb{q}-1)}
-\frac{\eb{q}(1-\embs)(\eb{(r+s)}-1)}{(\eb{r}-1)(\eb{q}-1)^2}\right].
\end{multline*}

Since $B_2=B_4$, this gives:
\begin{align*}
\rE[\cl_\varepsilon\cl^s_\varepsilon]=
B_1-B_3&=2\embs\int_\varepsilon^\infty dr \int_s^{s+r} dq \frac{\eb{q}}{(\eb{r}-1)(\eb{q}-1)}\\
&\hspace{1cm}
+2\ebs\int_\varepsilon^\infty dr \int_{s+r}^\infty dq
\frac{\eb{r}}{(\eb{r}-1)(\eb{q}-1)}\\
&\hspace{1cm}
-2\ebs\int_\varepsilon^\infty dr
\int_s^\infty dq \frac{1}{(\eb{q}-1)(\eb{r}-1)}\cdot
\end{align*}
Basic calculations yield:
\begin{multline*}
\rE[W_0W_s]=\frac{1}{2\beta^2\theta^2}
\left[\frac{\pi^2}{6}\embs +\ebs \Li_2(\embs)\right]\\
 -2(\ebs-\embs)\int_0^\infty dr \int_s^{s+r} \frac{dq}{(\eb{r}-1)(\eb{q}-1)}\cdot
\end{multline*}
The proof is then complete.

\subsection{Proof of Theorem \ref{theo:cadlag}}
\label{sec:proofcadlag}

   According to  Billingsley \cite[Theorem 3.16]{b:cpm} and thanks to the
   stationary property,  it is enough to check
   the following two conditions:
\begin{enumerate}
   \item[(i)] Right continuity in probability: for all $\lambda>0$:
\begin{equation}
   \label{eq:cond1}
\lim_{h \downarrow 0}
     \rP(|W_h -W_0|>\lambda)=0.
 \end{equation}
\item[(ii)] Control of the  jumps: there exists $\gamma>0$, $\delta>0$ such
  that for some constant $C>0$  and all $\lambda>0$, $s,t\in (0,1/8)$:
\begin{equation}
   \label{eq:cond2}
\rP(|W_{-t} -W_0| \wedge |W_s -W_0|\geq 6\lambda)
\leq  C\lambda ^{-4\gamma} (s+t)^{1+\delta}.
 \end{equation}
\end{enumerate}

Notice that Proposition \ref{prop:ewswo} (see also Remark \ref{rem:cov})
implies the $L^2$-continuity of
$W$. This in turn implies \reff{eq:cond1} and thus (i) is satisfied.

We shall now focus on (ii) and \reff{eq:cond2}.
In this Section $C$ denotes any finite positive constants which may
vary from line to line.
For notational convenience, we
shall write for $\varepsilon>0$, $s\in \R$:
\[
\int_0^\varepsilon dr\; (M_{s-r}^s - c(r) Z_s)= W_s - \cl_\varepsilon^s.
\]

We define for $h>|u|>0$:
\begin{align*}
   A_1(u,h)&=-\int_h^\infty (M_{-r} -M^u_{-r})dr,
\\
A_2(u,h)&=Z_0 \int_{h}^\infty c(r)\,
dr-Z_u \int_{h+u}^\infty  c(r) \, dr, \\
A_3(u,h)&=-\int_0^h (M_{-r} -c(r) Z_0)\; dr,\\
A_4(u,h)&=\int_0^{h+u} (M_{u-r}^u -c(r) Z_u)\; dr.
\end{align*}
For $s>0$, $t>0$ and $h>s+t$, we have:
\begin{equation}
   \label{eq:Ws-W0}
W_s-W_0=\sum_{i=1}^4 A_i(s,h)\quad \text{and}\quad
W_0-W_{-t}=-\sum_{i=1}^4 A_i(-t,h).
\end{equation}

In a first step we give upper bounds for the probability of $A_i$ to be
large in the following Lemmas.
\begin{lem}
   \label{lem:A1}
   There exists  a finite  constant $C_1$  such that for  all $s,  t \in
   (0,1/8)$, $h>2(s+t)>0$ and $\lambda>0$, we have:
\[
\rP(|A_1(s,h)|\wedge |A_1(-t,h)|>\lambda)\leq  C_1  \frac{(s+t)^2}{h^4}\cdot
\]
\end{lem}
\begin{proof}
Notice that $A_1(u,h)\neq 0$ implies  $|M_{-h} - M_{-h}^u|\geq
1$. Therefore, we have:
\begin{multline}
   \label{eq:majoA1A1}
   \rP(|A_1(s,h)|\wedge |A_1(-t,h)|>\lambda)\\
\begin{aligned}
&\leq
\rP( M_{-h} - M_{-h}^s\neq 0,\,  M_{-h} - M_{-h}^{-t} \neq 0)\\
&=1 - \rP( M_{-h} = M_{-h}^s) - \rP(  M_{-h} = M_{-h}^{-t} )
+\rP( M_{-h} = M_{-h}^s,\,  M_{-h} = M_{-h}^{-t} )\\
&=1 - \rP( M_{-h} = M_{-h}^s) - \rP(  M_{-h} = M_{-h}^{-t} )
+\rP(M_{-h}^s= M_{-h}^{-t} ),
\end{aligned}
\end{multline}
where we used for the last equality that the sequence $(M_{-h}^u, u>
-h)$ is non-increasing.

Let $r> t>0$. According to representation \reff{eq:repMr}, we have that $M_0^r$  is,   conditionally  on  $M_0^t$,  binomial  with
parameter $\left(M_0^t,\frac{c(r)}{c(t)}\right)$.  This implies:
\[
\rP(M_0^r=M_0^t)=\rE\left[\left(\frac{c(r)}{c(t)}\right)^{M_0^t}
\right] = \left(1+ \frac{c(t) -c(r)}{2\theta}\right)^{-2},
\]
where we used \reff{eq:lapm}
for the last equality. By stationarity, we deduce from
\reff{eq:majoA1A1} that:
\begin{align*}
 \rP(|A_1(s,h)|\wedge |A_1(-t,h)|>\lambda)
&   \leq 1 - \inv{(1+x)^2}
- \inv{(1+y)^2} + \inv{(1+x+y)^2} \\
& = \int_0^x dv \int_0^y  dz \, \frac{6}{(1+v+z)^4}\\
& \leq  6 xy,
\end{align*}
with
\[
x=\frac{c(h-t) -c(h)}{2\theta}\quad\text{and}\quad  y=\frac{c(h) -c(h+s)}{2\theta}\cdot
\]
Notice that:
\[
(1- \expp{-2\beta \theta h}) (\expp{2\beta \theta(h+s)} -1)\geq
\expp{2\beta \theta h} +\expp{-2\beta \theta h} -2 \geq
(2\beta \theta)^2 h^2.
\]
Since for $s\in (0,1/4)$, we have $\expp{ 2\beta \theta s} -1 \leq  Cs$,
we deduce that $y\leq Cs /h^2$. Similarly, and using $h-t\geq h/2$, we
also get $x\leq  Ct /h^2$. This implies:
\[
\rP(|A_1(s,h)|\wedge |A_1(-t,h)|>\lambda)\leq C \frac{st}{h^4} \leq  C
\frac{(s+t)^2}{h^4}\cdot
\]
\end{proof}

\begin{lem}
   \label{lem:A2}
 There exists a finite constant $C_2$ such that for all
$h>2|u|>0$, with $u\in [-1/8, 1/8]$, $\lambda>0$, we have:
\[
\rP(|A_2(u,h)|>\lambda)\leq  C_2 \frac{u^2}{\lambda^{4}h^4}\cdot
\]
\end{lem}

\begin{proof}
We write $A_2(u,h)=A_{2,1} + A_{2,2}$
with
\[
A_{2,1}= (Z_0-Z_u) \int _{u+h}^{+\infty } c(r)\, dr =
\frac{Z_u -Z_0}{\beta} \log(1-\expp{-2\beta\theta(h+u)})
\quad\text{and}\quad
A_{2,2}= Z_0 \int_h^{h+u} \!\! c(r)\, dr.
\]
We have:
\begin{equation}
   \label{eq:inegA2}
\rP(|A_2(u,h)|>\lambda)\leq \rP(|A_{2,1}|>\lambda/2)+
\rP(|A_{2,2}|>\lambda/2).
\end{equation}

Tchebychev's inequality gives:
\[
\rP(|A_{2,1}|>\lambda/2)\leq \frac{2^4}{(\lambda\beta)^{4}} \log(1-\expp{-2\beta\theta(h+u)})^{4}
\rE\left[|Z_u-Z_0|^{4}\right].
\]
Since $Z$ is a Feller diffusion, see Section 7.1 in
\cite{cd:spsmrcatsbp}, we have $\rE\left[|Z_u-Z_0|^{4}\right] \leq C
u^2$.
For $x>0$, we have $0\leq  -\log(1-\expp{-x})\leq  1/x$. Using that
$h>2|u|$, we get:
\begin{equation}
   \label{eq:inegA21}
\rP(|A_{2,1}|>\lambda/2)\leq C \frac{u^2}{\lambda^4(h+u)^4} \leq  C
\frac{u^2}{\lambda^4 h^4}.
\end{equation}

Since $c$ is decreasing and $h>2|u|$, we have:
\[
\int_h^{h+u} \!\! c(r)\, dr\leq  |u| c(h-|u|)\leq C \frac{|u|}{h}\leq C
\frac{\sqrt{|u|}}{h}\cdot
\]
Then, using Tchebychev's inequality, we get:
\begin{equation}
   \label{eq:inegA22}
\rP(|A_{2,1}|>\lambda/2)\leq C \frac{u^2}{\lambda^{4} h^4} \cdot
\end{equation}

Then use \reff{eq:inegA2} with \reff{eq:inegA21} and \reff{eq:inegA22}
to conclude.
\end{proof}

\begin{lem}
   \label{lem:A3}
 Let $p\in \N^*$. There exists a finite constant $C_3$ such that for all  $h>0$ and $\lambda>0$, we have:
\[
\rP(|A_3(u,h)|>\lambda)\leq  C_3\frac{h^p}{\lambda^{2p}}\cdot
\]
\end{lem}

\begin{proof}
Using  $\rE\left[\prod_{k=1}^{2p}|X_k|\right]\leq  \prod _{k=1}^{2p}
\rE \left[X_k^{2p}\right]^{1/2p}$, we get:
\[
\rE\left[A_3(u,h)^{2p}\right]
\leq
\left(\int_0^h
 dr\,
\rE \left[(M_{-r} -c(r) Z_0)^{2p}\right]^{1/2p}
\right)^{2p}.
\]
For a Poisson random variable $X$ with mean $m$, we have:
\[
\rE\left[(X-m)^{2p}\right] \leq  C' (m^p+m),
\]
where the constant $C'$ doesn't depend on $m$.
Thanks to Lemma \ref{lem:MZ}, we have that $M_{-r}$ is, conditionally on
$Z_0$, a Poisson random variable with mean $c(r) Z_0$. This implies that:
\[
\rE\left[(M_{-r} -c(r) Z_0)^{2p}\right]^{1/2p}\leq C (\sqrt{c(r)}+c(r)^{1/2p}) .
\]
We deduce that:
\[
\int_0^h
 dr\,
\rE \left[(M_{-r} -c(r) Z_0)^{2p}\right]^{1/2p}\leq  C (\sqrt{h} \wedge
1) \leq  C \sqrt{h}.
\]
Therefore, we have $\rE\left[A_3(u,h)^{2p}\right]\leq  Ch^p$ and we
conclude using Tchebychev's inequality.
\end{proof}

\begin{lem}
   \label{lem:A4}
 Let $p\in \N^*$. There exists a finite constant $C_4$ such that for all
 $h>2|u|$ and $\lambda>0$, we have:
\[
\rP(|A_4(u,h)|>\lambda)\leq  C_4\frac{h^p}{\lambda^{2p}}\cdot
\]
\end{lem}
\begin{proof}
By stationarity, we have that  $A_4(u, h)$ is distributed as
$-A_3(u,h+u)$. Then use Lemma \ref{lem:A3} to conclude.
\end{proof}

We  complete   the  proof   of  Theorem  \ref{theo:cadlag}   by  proving
\reff{eq:cond2}.  Using:
\[
\{ |x+y| \wedge |x'+y'|>6\lambda\}
\subset
\{|x|\wedge |x'|>3\lambda\} \bigcup \{|y|> 3\lambda\} \bigcup \{|y'|>3\lambda\} ,
\]
we get:
\begin{multline}
   \label{eq:majoPAB}
\rP(|W_{-t} -W_0| \wedge |W_s -W_0|\geq 6 \lambda)
\\\leq  \rP(|A_1(s,h)|\wedge |A_1(-t,h)|>3\lambda)
+ \sum_{i=2}^4 \rP(|A_i(s,h)|>\lambda)
+ \sum_{i=2}^4 \rP(|A_i(-t,h)|>\lambda).
\end{multline}

 Let  $\delta\in   (0,1/3)$,  $2>\gamma >
3(1+\delta)/2$ and $p\in \N^*$ such that $2p/(p+4)>  \gamma$. Notice
that $(1+\delta)/2\gamma< 1/3$.
Set:
\[
x=\frac{\lambda^{4\gamma}}{(s+t)^{1+\delta}} \quad \text{and}\quad
h=(s+t)^{(1+\delta)/2\gamma} \, x^{1/4}.
\]

If $x<1$, then we have that \reff{eq:cond2} holds trivially with
$C=1$. So we shall assume that $x\geq 1$. For $s,t\in (0, 1/8)$,
we have:
\[
h\geq (s+t)^{(1+\delta)/2\gamma}\geq (s+t)^{1/3}> 2 (s+t).
\]
So hypothesis of the previous Lemmas are satisfied for $s,t\in (0,
1/8)$. Using $\gamma-(1+\delta)>0$,
 Lemma \ref{lem:A1} implies:
\[
\rP(|A_1(s,h)|\wedge |A_1(-t,h)|>3 \lambda)
\leq  C_1 \frac{(s+t)^2}{h^4}
=  C_1 \frac{(s+t)^{\frac{2}{\gamma}(\gamma-(1+\delta))}}{x} \leq \frac{C_1}{x}\cdot
\]
For $u\in \{-t, s\}$, using
$2\gamma- 3(1+\delta)>0$, Lemma \ref{lem:A2} implies:
\[
\rP(|A_2(u,h)|> \lambda)
\leq  C_2 \frac{u^2}{\lambda^4 h^4} \leq   C_2 \frac{(s+t)^2}{\lambda^4
  h^4}
= C_2\frac{(s+t)^{2 - 3(1+\delta)/\gamma} }{x^{1+1/\gamma}}
\leq  \frac{C_2}{x}\cdot
\]
For $u\in \{-t, s\}$, using  $p(2-\gamma) \geq 8p/(p+4) >4\gamma$
together with Lemma \ref{lem:A3} resp. Lemma \ref{lem:A4}, we get:
\[
\rP(|A_3(u,h)|> \lambda)
\leq  C_3 \frac{h^p}{\lambda^{2p}}
=  \frac{C_3}{x^{\frac{p}{4\gamma}(2 -\gamma)} }
\leq  \frac{C_3}{x} \cdot
\]
resp. for $u\in \{-t, s\}$:
\[
\rP(|A_4(u,h)|> \lambda)
\leq  C_4 \frac{h^p}{\lambda^{2p}}
\leq  \frac{C_4}{x} \cdot
\]
We deduce that for $s,t\in (0, 1/8)$, and $x\geq 1$, we have:
\[
\rP(|W_{-t} -W_0| \wedge |W_s -W_0|\geq 6 \lambda)
\leq  \frac{C}{x}\cdot
\]
This ends the proof of \reff{eq:cond2} and thus (ii).

\section{Appendix}

\subsection{Functionals of the number of ancestors for the process $Y$}
\label{sec:app-R}

We have  the following results.

\begin{lem}
   \label{lem:Rab0}
Let $\lambda,  v, q\in (0,+\infty )$. We have:
\begin{equation}
   \label{eq:nexpR}
\N\left[1- \expp{-\lambda R_v^{v+q}}\right]= u(c(q)(1-\expp{-\lambda}),v),
\end{equation}
\begin{equation}
\label{eq:nrabh}
\N[R_v^{v+q}]=c(q) \expp{-2\beta \theta v}, \quad \text{and} \quad
\N[R_v^{v+q}|\zeta>v+q]=\frac{c(q)}{c(v+q)}  \expp{-2\beta \theta v}.
\end{equation}
\end{lem}

\begin{proof}
Since $R_v^{v+q}$ is, conditionally on $Y_v$, a Poisson random variable
with parameter
$c(q)Y_v$, we get thanks to \reff{eq:defu}:
\[
\N\left[1- \expp{-\lambda R_v^{v+q}}\right]=
\N\left[1- \expp{-c(q) (1- \expp{-\lambda}) Y_v}\right]=
u(c(q)(1-\expp{-\lambda}),v).
\]
   Equalities \reff{eq:nrabh} are a consequence of \eqref{eq:nyt}
   and \eqref{eq:defc} and the equality
   $\N[R_v^{v+q}]=\N[Y_v]\N[\zeta\geq q]$.
\end{proof}

We shall need later on other closed formulas for the joint distribution of the
number of ancestors at different time. We first give (in a slightly more
general statement) the conditional distribution of $R_v^{v+q+s}$ knowing
$R_v^{v+q}$.

Let $v,q,s\in (0,+\infty )$. Notice that an ancestor at time $v$ of the population
at time $v+q$ is also an ancestor of the population at time $v+q+s$ with
probability $c(q+s)/c(q)$ and this happens independently of the other
ancestors of the population at times before $v+q$. We deduce the
following Corollary.

\begin{cor}
   \label{cor:binR}
   Let $v,q,s\in (0,+\infty )$.  Conditionally on $(R_u^h; \, u\in
   (0,v], h\in (u,v+q])$,  the   random variable   $R_v^{v+q+s}$  has under  $\N$  a  binomial
   distribution with parameter $(R_v^{v+q}, c(q+s)/c(q))$.
\end{cor}

We recall the decomposition of $H$ before  and after  $T_b=\inf\{t;
H(t)=b\}$   under   $\N$ in  order to give the
conditional distribution of $R_{v+r}^{v+q}$ knowing $R_v^{v+q}$.

On   $\{T_b<+\infty   \}$, let
$((\alpha^\text{(g)}_i,   \beta^\text{(g)}_i),   i\in  I^\text{(g)}   )$
(resp.  $((\alpha^\text{(d)}_i,  \beta^\text{(d)}_i), i\in  I^\text{(d)}
)$) be the  excursion intervals  of $H$ above  its minimum backward  on the
left  of   $T_b$  (resp.  forward   on  the  right  of   $T_b$).  Define
$H^\text{(g)}_i$ for $i\in I^\text{(g)}$ as follows:
\[
H^\text{(g)}_i(t)=    H((t+\alpha_i^\text{(g)})\wedge    \beta_i^\text{(g)})  -
H(\alpha_i^\text{(g)}), \quad t\geq 0,
\]
and $H^\text{(d)}_i$ similarly for $i\in I^\text{(d)}$. It is well
known, see \cite{dlg:rtlpsbp} or \cite{ad:wdlcrtseppnm},
that under $\N[\cdot |\, T_b<+\infty ]$ the measures:
\[
\sum_{i\in I^\text{(g)}}
\delta_{(H(\alpha^\text{(g)}_i),\, H^\text{(g)}_i)} (dt, dH)
\quad\text{and}\quad
\sum_{i\in I^\text{(d)}}
\delta_{(H(\alpha^\text{(d)}_i),\, H^\text{(d)}_i)}(dt, dH)
\]
are independent Poisson point measures with respective intensity:
\[
\ind_{[0,b]}(t) \, \beta dt\N[dH; \max(H)< b-t]
\quad\text{and}\quad
\ind_{[0,b]}(t) \, \beta dt\N[dH].
\]

Let $v,q\in (0,+\infty )$.  By considering the $R_v^{v+q}$ excursions of
$H$ above level $v$ which reach level $v+q$ and the previous
representation for each of those excursions (with $b=q$), we easily deduce the
following result.

\begin{prop}
   \label{prop:Y|R}
   Let $v,q \in (0,+\infty  )$. Conditionally on $(R_v^{v+q}, Y_v)$, the
   process $(Y_{t+v}, t\in [0,q])$ is distributed under $\N$ as $(\tilde
   Y_t, t\in [0,q])$ with:
\[
\tilde Y_t= Y'_t + \sum_{t_i^\text{(g)}\leq t} Y^{\text{(g)},
  i}_{t-t_i^\text{(g)} }+  \sum_{t_i^\text{(d)}\leq t} Y^{\text{(d)},
  i}_{t-t_i^\text{(d)} }
\]
where $Y'$ is distributed according to $\P_{Y_v}(\cdot | \zeta<q)$,  $\sum_{i\in I^\text{(g)}}
\delta_{(t_i^\text{(g)}, Y^{\text{(g)},i})}$ and $\sum_{i\in I^\text{(d)}}
\delta_{(t_i^\text{(d)}, Y^{\text{(d)},i})}$
are independent Poisson point measures independent of $Y'$ with respective
intensity:
\[
\ind_{[0,q]}(t)\, \beta R_v^{v+q}  \N[dY; \zeta<q-t]
\,dt
\quad  \text{and}\quad
\ind_{[0,q]}(t)\, \beta R_v^{v+q}  \N[dY]
\,dt.
\]
\end{prop}

We deduce the following Corollary on the conditional distribution of
$R_{v+r}^{v+q}$ knowing $R_v^{v+q}$.

\begin{cor}
   \label{cor:R|R}
Let $\lambda, v,q, r \in (0,+\infty )$ with $q>r$. We have:
\[
\N\left[\exp \left( -\lambda R_{v+r}^{v+q} \right)\big| R_v^{v+q}\right]= h(\lambda)^{R_v^{v+q}},
\]
with
\[
h(\lambda)=\expp{-\lambda} \left(1 - \frac{u\left(c(q-r)(1-\expp{-\lambda}), r\right)}{c(r)}\right).
\]
\end{cor}
\begin{proof}
 We have:
\begin{align*}
\N\left[\exp \left( -\lambda R_{v+r}^{v+q} \right)\big| R_v^{v+q}\right]
&=
\exp \left(-\lambda R_v^{v+q} - \beta R_v^{v+q} \int_0^r ds\,
  \N\left[1-\expp{-\lambda R_s^{s+q-r}} \right] \right)\\
&=
\exp \left(-\lambda R_v^{v+q} - \beta R_v^{v+q} \int_0^r ds\,
 u(c(q-r)(1-\expp{-\lambda}), s)\right)\\
&=    \expp{-\lambda R_v^{v+q}} \left(1+c(q-r)(1-\expp{-\lambda}) \frac{1-
    \expp{-2 \beta \theta r}}{2\theta}  \right)^{ - R_v^{v+q}} \\
&= h(\lambda)^{  R_v^{v+q}},
\end{align*}
where we used  Proposition \ref{prop:Y|R} for the first equality,
\reff{eq:nexpR} for the second and \reff{eq:int-u} for the third and
elementary computation for the last.
\end{proof}

We give the following elementary results which are  used  in Section \ref{sec:timereversal}.

\begin{lem}
   \label{lem:Rab}
Let $\lambda, \mu, v, q, s\in (0,+\infty )$. We have with $\kappa_1=\left(1-\expp{-\lambda}
\right)c(q)+\expp{-\lambda}c(q+s)$:
\begin{equation}
   \label{eq:nrabz1}
\N\left[\left(1-\expp{-\lambda R_v^{v+q}}\right)\ind_{\{\zeta<v+q+s\}}\right]
= u(\kappa_1, v)-c(v+q+s);
\end{equation}
for $0<v'<q$, with $\kappa_2=
\left(1-\expp{-\mu}
\right) c(q-v')+ \expp{-\mu} \kappa_1$:
\begin{equation}
   \label{eq:nrabz1.5}
\N\left[\left(1-\expp{-\lambda R_v^{v+q}- \mu
      R_v^{v+q-v'}}\right)\ind_{\{\zeta<v+q+s\}}\right]
= u(\kappa_2, v)-c(v+q+s);
\end{equation}
for $0<v'<v$ with $\kappa_3=(1-\expp{-\mu})c(q+v')+\expp{-\mu} u(\kappa_1, v')$:
\begin{equation}
\label{eq:nrabz2}
 \N\left[\left(1-\expp{-\lambda R_{v}^{v+q}-\mu
       R_{v-v'}^{v+q}}\right)\ind_{\{\zeta<v+q+s\}}\right]=u(\kappa_3,
 v-v')-c(v+q+s);
\end{equation}
 and
\begin{equation}
   \label{eq:nrabz}
\N\left[R_v^{v+q}\ind_{\{\zeta<v+q+s\}}\right]=\left(c(q) -c(q+s)\right)\expp{2\beta \theta v}
\left(\frac{c(v+q+s)}{c(q+s)}\right)^2.
\end{equation}
\end{lem}

\begin{proof}
We have:
\begin{align*}
\N\left[\left(1-\expp{-\lambda R_v^{v+q}}\right)\ind_{\{\zeta<v+q+s\}}\right]
&= \N\left[\left(1-\expp{-\lambda R_v^{v+q}}\right)\left(1-
    \frac{c(q+s)}{c(q)} \right)^{R_v^{v+q}} \right]\\
&= u\left(c(q)  \left(1- \expp{-\lambda} \left(
1-    \frac{c(q+s)}{c(q)}
\right)\right), v \right) - u\left(c(q+s), v\right),
\end{align*}
where we used   $\{\zeta<v+q+s\}=\{
R_v^{v+q+s}=0\}$ and Corollary \ref{cor:binR} for the first equality and
\reff{eq:nexpR} twice for the second.
Then use \reff{eq:uc=c}
to get \reff{eq:nrabz1}.

The proof of \reff{eq:nrabz1.5} relies on the same type of arguments and
is left to the reader.

Taking the derivative with respect to $\lambda$
at $\lambda=0$ in \reff{eq:nrabz1} gives \reff{eq:nrabz}.

We prove \reff{eq:nrabz2}.
Set $\expp{-\tilde \lambda}=\expp{-\lambda} \left(1-
    \frac{c(q+s)}{c(q)} \right)=1- \frac{\kappa_1}{c(q)}$. We have:
\begin{align*}
   \N\left[\left(1-\expp{-\lambda R_{v}^{v+q}-\mu
       R_{v-v'}^{v+q}}\right)\ind_{\{\zeta<v+q+s\}}\right]
&=  \N\left[\left(1-\expp{-\lambda R_{v}^{v+q}-\mu
       R_{v-v'}^{v+q}}\right)\left(1-
    \frac{c(q+s)}{c(q)} \right)^{R_v^{v+q}} \right]\\
&= \N\left[1-\expp{-\tilde \lambda R_{v}^{v+q}-\mu
       R_{v-v'}^{v+q}} \right] - u(c(q+s),v) \\
&=\N\left[1-(h_1(\lambda)\expp{-\mu})^{R_{v-v'}^{v+q}}\right] - c(v+q+s) ,
\end{align*}
where we used  $\{\zeta<v+q+s\}=\{
R_v^{v+q+s}=0\}$ and  Corollary \ref{cor:binR} for the first equality,
\reff{eq:nexpR}  for the second, \reff{eq:uc=c}
and Corollary
\ref{cor:R|R} for the third with:
\[
h_1(\lambda)=\expp{-\tilde \lambda} \left(1- \frac{u(c(q)(1-\expp{-\tilde
        \lambda}),v')}{c(v')}\right)
=\left(1-\frac{\kappa_1}{c(q)}\right)
\left(1-\frac{u(\kappa_1,v')}{c(v')}\right).
\]
Then use \reff{eq:nexpR} to get $
\N\left[1-(h_1(\lambda)\expp{-\mu})^{R_{v-v'}^{v+q}}\right]=u(h_2(\mu),v-v')$
with
\[
h_2(\mu)=c(q+v')\left(1-\expp{-\mu}h_1(\lambda) \right).
\]
Taking $\mu=0$, we get using \reff{eq:nexpR} and \reff{eq:uc=c}:
\begin{align*}
u(h_2(0),v-v')&=\N\left[\left(1-\expp{-\lambda
      R_{v}^{v+q}}\right)\ind_{\{\zeta<v+q+s\}}\right]+ c(v+q+s)\\
&= u(\kappa_1,v)\\
&=u(u(\kappa_1, v'), v-v').
\end{align*}
We deduce that $h_2(0)=u(\kappa_1, v')$ and since
\[
h_2(\mu)=c(q+v')(1-\expp{-\mu}) + \expp{-\mu} h_2(0),
\]
we get    \reff{eq:nrabz2}.
\end{proof}

\subsection{Moments of the number of ancestors for the process $Z$}
\label{sec:MomZ}

We easily get the following result using Proposition \ref{prop:M-mart}.
\begin{cor}
   \label{cor:EMtMr}
Let $r\geq t>0$. We have:
\begin{equation}
   \label{eq:MrdMt}
\rE\left[M_{-t}|M_{-r}\right]= \frac{c(t)}{\theta} \left(1-
  \expp{-2\beta\theta(r-t)}\right) + \frac{c(t)}{c(r)}
\expp{-2\beta\theta(r-t)} M_{-r}
\end{equation}
and
\begin{equation}
   \label{eq:MrMt}
\rE\left[M_{-t}M_{-r}\right]= \frac{c(r)}{\theta}\left(1+
  \frac{3}{2}\frac{c(t)}{\theta}\right)= 2\frac{\expp{2\beta\theta
       t}+2}{(\expp{2\beta\theta t} -1)(\expp{2\beta\theta r} -1)} \cdot
\end{equation}
\end{cor}

\begin{proof}
   Let $g(t)=\rE\left[M_{-t}|M_{-r}\right]$ and
   $h(t)=\rE\left[M_{-t}\right]=c(t)/\theta$. Using Proposition
   \ref{prop:M-mart}, we get that for $r\geq t>0$:
\[
g(t)=M_{-r}+\int^r_t \beta c(s) (g(s)+2)\; ds
\quad\text{and}\quad
h(t)=h(r)+\int^r_t \beta c(s) (h(s)+2)\; ds .
\]
This implies that $g'(t)-h'(t)=-\beta c(t)(g(t)-h(t))$ and thus:
\[
g(t)-h(t)=\left(M_{-r}-\rE\left[M_{-r}\right]\right)\expp{\beta\int_t^r
  c(s)\; ds}.
\]
Then use \reff{eq:intc} and \reff{eq:EM} to get \reff{eq:MrdMt}.

Taking the expectation in \reff{eq:MrdMt} and using the second part of
\reff{eq:EM}, we get:
\begin{align*}
\rE\left[M_{-t}M_{-r}\right]
&=\frac{c(t)c(r)}{\theta^2} \left(1-
  \expp{2\beta\theta (t-r)}\right)+ \frac{c(t)c(r)}{\theta^2}
\expp{2\beta\theta (t-r)} \left(\frac{\theta}{c(r)} + \frac{3}{2}\right)\\
&=    \inv{2} \frac{c(t)c(r)}{\theta^2}
\left(\expp{2\beta\theta t }+2\right)\\
&= \frac{c(r)}{\theta}\left(1+ \frac{3}{2}\frac{c(t)}{\theta}\right).
\end{align*}
\end{proof}

The following  Lemma  generalizes  \reff{eq:MrMt},
and is  used  in the proof of Proposition \ref{prop:ewswo}.
\begin{lem}\label{lem:m0ms}
Let $r>0$, $s>0$ and $q>0$. For $s+r\geq q$, we have:
\begin{equation}
   \label{eq:MrMsq}
\rE [M_{-r} M_{s-q}^s]=\frac{c(r)}{\theta} \frac{c(q)}{\theta}
\left(\frac{\theta}{c(q-s)}  +\frac{3}{2}\right).
\end{equation}
For $q\geq s+r$, we have:
\begin{equation}
   \label{eq:MrMsq+}
\rE [M_{-r} M_{s-q}^s]=\frac{c(r)}{\theta} \frac{c(q)}{\theta}
\left(\frac{\theta}{c(q-s)}\frac{c(q) }{c(r+s)}  +\frac{3}{2}\right).
\end{equation}
\end{lem}

\begin{proof}
First, we consider the case $s+r\geq q$. Given $M_{-r}$,
$M^s_{s-q}$ can be decomposed in two parts:
\[
M_{s-q}^s = M^{I[-r, s-q]} + \sum_{j = 1}^{M_{-r}}\tilde{M}_{s+r-q}^{(s+r), j},
\]
where $M^{I[-r, s-q]}$ is the number of ancestor coming from the
immortal individual  on the interval $(-r, s-q)$, and $\tilde
M_{s+r-q}^{(s+r), j}$
represents the number of ancestor
  generated by one of the ancestors at time $-r$. More precisely, we have
\begin{enumerate}
\item[(i)] $M^{I[-r,s -q]}$ is the  number of ancestors at time $s-q$ of
  the  population  living  at  time  $s$  corresponding  to  all  the
  populations  $Y^i$ (see  definition \reff{eq:defcn})  with immigration
  time  $t_i$  belonging  to  $(-r,s  -q)$  and  $M^{I[-r,  s  -q]}$  is
  independent of $M_{-r}$.
\item[(ii)] $(\tilde M_{s+r-q}^{(s+r), j}, j\in \N^*)$ are independent,
  independent of $M_{-r}$ and are  distributed as
  $R_{s+r-q}^{s+r}$ under $\N[dY|\zeta>r]$.
\end{enumerate}
We deduce that:
\[
\rE[M_{s-q}^s|M_{-r}]= 2\beta \int_0^{s+r-q} \N[R_t^{t+q}]\, dt + M_{-r}
\N[R_{s+r-q}^{s+r} |\zeta>r].
\]
Using \reff{eq:nrabh} and \reff{eq:MrMt}, elementary computations give:
\[
\rE [M_{-r} M_{s-q}^s]=\rE\left[M_{-r}\rE[M_{s-q}^s|M_{-r}]\right]
=\frac{2(\eb{(q-s)}+2)}{(\eb{q}-1)(\eb{r}-1)}\cdot
\]
This gives \reff{eq:MrMsq}.

Second, we  consider the case  $q\geq s+r$.  Given $(  M_\ell^s, Z_\ell,
\ell\leq s-q)$,  the number of  ancestors $M_{-r}$ can be  decomposed in
three parts:
\[
M_{-r}=M^{I[s -q, -r]}+ \sum_{j=1}^{M^s_{s-q}} \tilde M_{q-s-r}^{(q-s),
  j}+ \sum_{i'} R_{q-s-r}^{(q-s),i'}(\hat Y^{i'}),
\]
where $M^{I[s -q, -r]}$ is the number of ancestor coming from the
immigration on the interval $( s-q,-r)$,  $\tilde M_{q-s-r}^{(q-s), j}$
represents the number of ancestor
  generated by one of the $M^s_{s-q}$ ancestors at time $s-q$, and
  $\hat Y^{i'}$ is a population generated from one of the individuals
  at time $s-q$ (among the population of size $Z_{s-q}$) which dies
  before  time $s$ (that is with lifetime less than $q$). More precisely, we have
\begin{enumerate}
   \item[(i)] $M^{I[s -q, -r]}$ is the  number of ancestors at time $-q$ of
  the  population  living  at  time  $0$  corresponding  to    all  the
  populations  $Y^i$ (see  definition \reff{eq:defcn})  with immigration
  time  $t_i$  belonging  to  $(s  -q, -r)$  and  $M^{I[  s  -q, -r]}$  is
  independent of $(  M_\ell^s, Z_\ell,
\ell\leq s-q)$.
\item[(ii)]  $(\tilde M_{q-s-r}^{(q-s), j}, j\in \N^*)$ are independent,
  independent of $(  M_\ell^s, Z_\ell,
\ell\leq s-q)$ and are  distributed as
  $R_{q-s-r}^{q-s}$ under $\N[dY|\zeta>q]$.
\item[(iii)] Conditionally on $(  M_\ell^s, Z_\ell,
\ell\leq s-q)$, $\sum_{i'} \delta_{\hat Y^{i'}}$ is a Poisson point
measure with intensity $Z_{s-q} \N[dY, \zeta< q]$.
\end{enumerate}
We deduce that:
\[
\rE [M_{-r}|\sigma( M_\ell^s, Z_\ell, \ell\leq s-q)]
= \rE[M^{I[s -q, -r]}] + M_{s-q}^s \N[R_{q-s-r}^{q-s}|\zeta>q]+
Z_{s-q} \N\left[R_{q-s-r}^{q-s}\ind_{\{ \zeta< q\}}\right] .
\]
We have:
\[
\rE[M^{I[s -q, -r]}]=2 \beta \int_{s-q}^{-r} dv\; \N[R_{-r-v}^{-v}]=
\frac{c(r)}{\theta}(1-\emb{(q-s-r)}).
\]
Using \reff{eq:nrabz}, we get:
\[
\N\left[R_{q-s-r}^{q-s}\ind_{\{\zeta< q\}}\right]
=(c(r)-c(r+s))
\frac{\psi(c(q))}{\psi(c(s+r))},
\]
as well as:
\begin{align*}
\N[R_{q-s-r}^{q-s}|\zeta>q]
&=c(q)^{-1}\N[R_{q-s-r}^{q-s}\ind_{\{\zeta> q\}}]\\
&=\frac{c(r)}{c(q)}\emb{(q-s-r)}-\frac{c(r)-c(r+s)}{c(q)}
\frac{\psi(c(q))}{\psi(c(s+r))}\cdot
\end{align*}
Then elementary computation yields:
\begin{align*}
\rE [M_{-r} M_{s-q}^s]
&=\rE \left[M_{s-q}^s
\rE [M_{-r}|\sigma( M_\ell^s, Z_\ell, \ell\leq s-q)]\right]\\
&=\frac{2(\eb{(r+s)}+2)}{(\eb{r}-1)(\eb{q}-1)}
-\frac{c(r)-c(r+s)}{\theta}\frac{\psi(c(q))}{\psi(c(s+r))}\cdot
\end{align*}
Then it is straightforward to get the desired result.
\end{proof}

\end{document}